\documentclass[11pt]{amsart}
\usepackage{latexsym}
\usepackage{amssymb}
\usepackage{amsmath}

\usepackage{amsfonts}

\newtheorem{theorem}{Theorem}[section]
\newtheorem{lemma}[theorem]{Lemma}
\newtheorem{proposition}[theorem]{Proposition}
\newtheorem{corollary}[theorem]{Corollary}
\newtheorem{definition}[theorem]{Definition}
\newtheorem{example}[theorem]{Example}

\newtheorem{question}[theorem]{Question}

\newcommand\nph{\varphi}

\newcommand\Proj{\mathop{\rm Proj}}

\newcommand\id{\mathop{\rm id}}
\newcommand\Ref{\mathop{\rm Ref}}
\newcommand\ran{\mathop{\rm Ran}}
\newcommand\Ball{\mathop{\rm Ball}}

\newcommand{\cl}[1]{\mathcal{#1}}
\newcommand{\bb}[1]{\mathbb{#1}}

\newcommand{\sca}[1]{\left\langle#1\right\rangle} %\sca{a,b}
\begin{document}

\title{Ranges of bimodule projections and reflexivity}

\author{G. K. Eleftherakis and I. G. Todorov}
\date{18 August 2011}
\address{Department of Mathematics, University of Athens,
Panepistimioupolis 157 84, Athens, Greece}

\email{gelefth@math.uoa.gr}

\address{Department of Pure Mathematics, Queen's University Belfast,
Belfast BT7 1NN, United Kingdom}

\email{i.todorov@qub.ac.uk}

\begin{abstract}
We develop a general framework for reflexivity in dual Banach
spaces, motivated by the question of when the weak* closed linear
span of two reflexive masa-bimodules is automatically reflexive. We
establish an affirmative answer to this question in a number of
cases by examining two new classes of masa-bimodules, defined in
terms of ranges of masa-bimodule projections. We give a number of
corollaries of our results concerning operator and spectral
synthesis, and show that the classes of masa-bimodules we study are
operator synthetic if and only if they are strong operator Ditkin.
\end{abstract}

\maketitle

\section{Introduction}

Operator synthesis, introduced by W. Arveson \cite{a} and
subsequently developed by V.S. Shulman and L. Turowska \cite{st1},
\cite{st2}, is an operator theoretic version of the well-known
concept of spectral synthesis in Harmonic Analysis. Due to the work
of W. Arveson, J. Froelich, J. Ludwig, N. Spronk and L. Turowska
\cite{a}, \cite{f}, \cite{st}, \cite{lt}, it is known that the
notion of spectral synthesis \lq\lq embeds'' into that of operator
synthesis in that, for a large class of locally compact groups $G$,
given a closed subset $E$ of $G$, there is a canonical way to
produce a subset $E^*$ of the direct product $G\times G$, so that
the set $E$ satisfies spectral synthesis if and only if the set
$E^*$ satisfies operator synthesis. Thus, the well-known, and still
open, problem of whether the union of two synthetic sets is
synthetic can be viewed as a special case of the problem asking
whether the union of two operator synthetic sets is operator
synthetic.

The notion of operator synthesis is closely related to that of reflexivity. Recall that
a subspace $\cl S$ of the space $\cl B(H_1,H_2)$ of all bounded linear operators from a Hilbert space $H_1$
into a Hilbert space $H_2$ is called \emph{reflexive} if it coincides with its \emph{reflexive hull} \cite{ls}
$$\Ref\cl S = \{T\in \cl B(H_1,H_2) : Tx \in \overline{\cl S x}, \ \mbox{ for all } x\in H_1\}.$$
Reflexive spaces are automatically closed in the weak* (and even the weak operator) topology.
In the present paper we initiate the study of the following question:

%\bigskip

%\noindent {\bf Question }
\begin{question}\label{q_rq}
Given two reflexive spaces $\cl S, \cl T\subseteq \cl B(H_1,H_2)$, when is the weak* closure
$\overline{\cl S + \cl T}^{w^*}$ of their sum reflexive?
\end{question}
%\bigskip

Question \ref{q_rq} is closely related to the
question of whether the union of two operator synthetic sets is operator synthetic.
Indeed, an affirmative answer to Question \ref{q_rq},
in the case $\cl S$ and $\cl T$ are bimodules over maximal abelian selafdjoint algebras
(masa-bimodules for short) with operator synthetic supports,
implies that the union of these supports is operator synthetic.

We obtain an affirmative answer to Question \ref{q_rq} in a number
of cases. Crucial for our considerations is the class of
masa-bimodules consisting of all ranges of weak* continuous bimodule
projections. The latter maps have attracted considerable attention
in the literature, as they are precisely the idempotent Schur
multipliers (see \cite{kp}). We study a class of masa-bimodules,
which we call \emph{approximately $\frak{I}$-injective}
masa-bimodules, that are defined as the intersections of sequences
of ranges of uniformly bounded weak* continuous masa-bimodule
projections, as well as the more general class of
\emph{$\frak{I}$-decomposable} masa-bimodules (Definition
\ref{d_ad}). Our most general result concerning Question \ref{q_rq}
is that it has an affirmative answer when $\cl S$ is a reflexive
masa-bimodule, while $\cl T$ is the intersection of finitely many
$\frak{I}$-decomposable masa-bimodules. In particular,
$\overline{\cl S + \cl T}^{w^*}$ is reflexive whenever $\cl S$ is
reflexive and $\cl T$ is a masa-bimodule (or a CSL algebra) of
finite width. These results are given as an application of a more
general result obtained in Section \ref{s_gf}, where a new reflexive
hull in the setting of dual Banach spaces is introduced and
examined. We hope that this general setting may be applied in other
instances as well.

Sections \ref{s_syn} and \ref{s_iam} are devoted to connections with
spectral and operator synthesis. As a corollary of our results, we
show that the union of an operator synthetic set and a set of finite
width is operator synthetic. This extends the results of \cite{st1}
and \cite{t_jlms}, where it was shown that sets of finite width are
operator synthetic. We give some applications concerning unions of
sets of spectral synthesis in locally compact groups. We show that
the supports of the ranges of weak* continuous masa-bimodule
projections are always operator synthetic. While we do not know
whether the same holds for the supports of approximately
$\frak{I}$-injective ones, we show that these sets satisfy a weaker
form of operator synthesis (see Theorem \ref{th_decai}). Moreover,
we show that the supports of the (more general)
$\frak{I}$-decomposable masa-bimodules are operator synthetic if and
only if they are strong operator Ditkin. We note that it is an open
question in Harmonic Analysis (resp. Operator Theory) whether every
synthetic set (resp. every operator synthetic set) is necessarily
Ditkin (resp. operator Ditkin).

In Section \ref{s_convth} we address the converse to Question \ref{q_rq},
and obtain sufficient conditions which ensure the reflexivity of $\cl T$, provided
$\cl S$ and $\overline{\cl S + \cl T}^{w^*}$ are both reflexive.

\section{The general framework}\label{s_gf}

In this section we set up the general framework, introducing a
reflexive hull for subspaces of a dual Banach space relative to a
family of commuting weak* continuous idempotents. Let $\cl X$ be a
dual Banach space with a fixed predual $\cl X_*$, and $\cl B(\cl X)$
be the space of all bounded linear operators on $\cl X$. As usual,
we denote by $\|\phi\|$ the norm of an operator $\phi \in \cl B(\cl
X)$. An \emph{idempotent} in $\cl B(\cl X)$ is an element $\phi \in
\cl B(\cl X)$ such that $\phi^2 = \phi$. If $\cl S\subseteq \cl X$,
for the rest of this section we will denote by $\overline{\cl S}$
the closure of $\cl S$ in the weak* topology arising from the
identity $\cl X = (\cl X_*)^*$, and by referring to a \lq\lq closed
set'' (resp. \lq\lq closed subspace''), we will always mean a set
(resp. a subspace), closed in the weak* topology of $\cl X$. The
convergence of nets of elements of $\cl X$ will also always be with
respect to the weak* topology. Closures in the norm topology of $\cl
X$ will not appear explicitly in the paper.

Let $\frak{C} \subseteq \cl B(\cl X)$ be a Boolean algebra of pairwise commuting
weak* continuous idempotents
with top element the identity operator $\id$ and bottom element the zero operator.
This means that $\frak{C}$ is closed under complementation (that is, $\phi \in \frak{C}$ implies
$\phi^{\perp} \stackrel{def}{=} \id - \phi\in \frak{C}$), contains $0$ and $\id$,
and $\phi + \psi - \phi\psi, \phi\psi\in \frak{C}$ whenever $\phi,\psi \in \frak{C}$.
We denote by $\ran\phi$ the range of an idempotent $\phi$.
It is easy to verify that if $\phi,\psi\in \frak{C}$ then $\ran\phi \cap \ran\psi = \ran (\phi\psi)$ and
$\ran\phi + \ran\psi = \ran (\phi + \psi - \phi\psi)$.

\begin{definition}\label{d_ge}
Let $\cl Y\subseteq \cl X$ be a linear subspace.

(i) \ \ The subspace $\cl Y\subseteq \cl X$  will be called \emph{$\frak{C}$-invariant} if $\phi(\cl Y)\subseteq \cl Y$
for every $\phi \in \frak{C}$.

(ii) \ The subspace $\cl Y\subseteq \cl X$ will be called \emph{$\frak{C}$-injective} if
$\cl Y = \ran\phi$ for some $\phi \in \frak{C}$.

(iii) The \emph{$\frak{C}$-reflexive hull} of $\cl Y$ is the subspace
$$\Ref\mbox{}_{\frak{C}} \cl Y = \{x\in \cl X : \mbox{ if } \phi \in \frak{C} \mbox{ and } \phi(\cl Y) = \{0\} \mbox{ then } \phi(x) = 0\}.$$
The subspace $\cl Y$ is called \emph{$\frak{C}$-reflexive} if $\Ref_{\frak{C}}\cl Y = \cl Y$.
\end{definition}

We record some elementary properties of the notions introduced in Definition \ref{d_ge}.

\begin{proposition}\label{p_intr}
(i) \ If $\cl Y_1\subseteq \cl Y_2$ then $\Ref_{\frak{C}}\cl Y_1\subseteq \Ref_{\frak{C}}\cl Y_2$.

(ii) The intersection of any family of $\frak{C}$-reflexive subspaces is $\frak{C}$-reflexive.

(iii) If $\cl Y$ is a $\frak{C}$-invariant subspace then $\Ref_{\frak{C}}\cl Y$ is $\frak{C}$-invariant.

(iv) For every subspace $\cl Y\subseteq \cl X$, the
$\frak{C}$-reflexive hull $\Ref_{\frak{C}}\cl Y$ is a closed
subspace of $\cl X$.

(v) Every $\frak{C}$-injective space is closed and $\frak{C}$-invariant.
\end{proposition}
\begin{proof}
(i) is trivial.

(ii) Let $\cl Y_{\alpha}\subseteq \cl X$, $\alpha \in \bb{A}$, be a family of $\frak{C}$-reflexive subspaces
of $\cl X$. By (i),
$$\Ref\mbox{}_{\frak{C}} (\cap_{\alpha\in \bb{A}} \cl Y_{\alpha}) \subseteq\cap_{\alpha\in \bb{A}} \Ref\mbox{}_{\frak{C}}\cl Y_{\alpha} =
\cap_{\alpha\in \bb{A}} \cl Y_{\alpha}.$$ The converse inclusion is trivial.

(iii) Let $x\in \Ref_{\frak{C}}\cl Y$ and $\phi \in \frak{C}$. Suppose that $\psi\in \frak{C}$ annihilates
$\cl Y$. Then $\psi(\phi(x)) = \phi(\psi(x)) = \phi(0) = 0$; thus, $\phi(x) \in \Ref_{\frak{C}}\cl Y$.

(iv) Let $(x_{\alpha})\subseteq \Ref_{\frak{C}}\cl Y$ be a net
converging to $x$ and $\phi\in \frak{C}$ annihilate $\cl Y$. Then
$\phi(x_{\alpha}) = 0$ for each $\alpha$ and, by the continuity of
$\phi$, we have that $\phi(x)=0$; thus, $x\in \Ref_{\frak{C}}\cl Y$.

(v)
Let $\phi,\psi \in \frak{C}$. For every $x = \phi(x)$ we have $\psi(x) = \psi(\phi(x)) = \phi(\psi(x))\in \ran\phi$,
thus, $\psi(\ran \phi) \subseteq \ran\phi$. Thus, $\ran\phi$ is $\frak{C}$-invariant.
To show that $\ran\phi$ is close, suppose that $(x_{\alpha})\in \ran\phi$ is a net with $x_{\alpha}\rightarrow x$.
Then $x_{\alpha} = \phi(x_{\alpha}) \rightarrow \phi(x)$ and so $x = \phi(x) \in \ran\phi$.
\end{proof}

\begin{proposition}\label{p_sin}
Let $\cl Y\subseteq \cl X$ be a closed $\frak{C}$-invariant subspace and $\cl M\subseteq \cl X$ be a $\frak{C}$-injective subspace.

(i) \ \ The space $\cl M$ is $\frak{C}$-reflexive.

(ii) \ The algebraic sum $\cl Y + \cl M$ is closed.

(iii) We have $\Ref_{\frak{C}}(\cl Y + \cl M) = \Ref_{\frak{C}}(\cl
Y) + \cl M$. Thus, if $\cl Y$ is $\frak{C}$-reflexive then $\cl Y +
\cl M$ is $\frak{C}$-reflexive.
\end{proposition}
\begin{proof}
Let $\phi\in \frak{C}$ be an idempotent with range $\cl M$.

(i) Let $x\in \Ref_{\frak{C}}\cl M$.
By the assumptions on the family $\frak{C}$, we have that
$\phi^{\perp}\in \frak{C}$, and clearly $\phi^{\perp}(\cl M) = \{0\}$. It follows that $\phi^{\perp}(x) = 0$,
that is, $x = \phi(x)\in \ran\phi = \cl M$.

(ii) Suppose that $(y_{\alpha})\subseteq \cl Y$ and $(m_{\alpha})\subseteq \cl M$ are nets such that
$y_{\alpha} + m_{\alpha}\rightarrow x$. The invariance of $\cl Y$ implies
$$x -\phi(x) = \lim\mbox{}_{\alpha} (y_{\alpha} + m_{\alpha} - \phi(y_{\alpha}) - \phi(m_{\alpha}))
= \lim\mbox{}_{\alpha} (y_{\alpha} - \phi(y_{\alpha}))\in \cl Y.$$
Thus, $x = (x-\phi(x)) + \phi(x)\in \cl Y + \cl M$.

(iii) Let $x\in \Ref_{\frak{C}}(\cl Y + \cl M)$.
We will show that $x - \phi(x) \in \Ref_{\frak{C}}\cl Y$.
Suppose that $\psi\in \frak{C}$ annihilates $\cl Y$.
Then $\phi^{\perp}\psi$ annihilates $\cl Y + \cl M$; indeed,
if $y\in \cl Y$ and $m\in \cl M$ then
$$\phi^{\perp}\psi(y+m) = \phi^{\perp}\psi(y) + \psi\phi^{\perp}(m) = 0.$$
Since the idempotents in $\frak{C}$ pairwise commute,
$$\phi(\psi(x-\phi(x))) = \phi\psi(x) - \phi\psi(x) = 0.$$
On the other hand, since $x\in \Ref_{\frak{C}}(\cl Y + \cl M)$ and
$\phi^{\perp}\psi$ annihilates $\cl Y + \cl M$, we have that
$$\phi^{\perp}(\psi(x-\phi(x))) = \phi^{\perp}\psi(x) = 0.$$
It follows that
$$\psi(x-\phi(x)) = \phi(\psi(x-\phi(x))) + \phi^{\perp}(\psi(x-\phi(x))) = 0.$$
Thus, $x - \phi(x)\in \Ref_{\frak{C}}\cl Y$ and hence
$x = (x-\phi(x)) + \phi(x) \in \Ref_{\frak{C}}\cl Y + \cl M$.
The second statement in (iii) is now immediate.
\end{proof}

\begin{definition}\label{d_ai}
A subspace $\cl M\subseteq \cl X$ will be called \emph{approximately $\frak{C}$-injective}
if there exists a sequence $(\phi_n)_{n\in \bb{N}}\subseteq \frak{C}$
and a constant $C > 0$ such that $\|\phi_n\|\leq C$, $\ran\phi_{n+1}\subseteq \ran\phi_n$, $n\in \bb{N}$,
and $\cl M = \cap_{n=1}^{\infty} \ran\phi_n$.
We call any sequence $(\phi_n)_{n=1}^{\infty}$ with these properties an \emph{associated sequence of
idempotents for $\cl M$}.
\end{definition}

\noindent {\bf Remarks (i) } Let $\cl M$ and $(\phi_n)_{n\in
\bb{N}}$ be as in Definition \ref{d_ai}. Since $\cl B(\cl X)$ is
itself a dual Banach space (see, for example, paragraph A.3.3 of
\cite{blm}), the sequence $(\phi_n)_{n\in \bb{N}}$ has a cluster
point $\phi\in \cl B(\cl X)$ in the point-weak* topology. It is
easily seen that $\phi$ is a (not necessarily weak* continuous)
idempotent with range $\cl M$. Moreover, since the ranges
$\ran\phi_n$ are nested, it follows that any (weak*) cluster point
of $(\phi_n(x))_{n=1}^{\infty}$ lies in $\cl M$.

\smallskip

\noindent {\bf (ii) }
Let $\cl M_j$, $j = 1,\dots,k$, be approximately $\frak{C}$-injective spaces. Then
$\cl M \stackrel{def}{=} \cap_{j=1}^k \cl M_j$ is approximately $\frak{C}$-injective. To see this,
let $(\phi_n^j)_{n=1}^{\infty}$ be an associated sequence of
idempotents for $\cl M_j$, $j = 1,\dots,k$. It is easy to see that the sequence
($\phi_n^1\phi_n^2\cdots\phi_n^k)_{n\in \bb{N}}$ is uniformly bounded and the intersection of
the ranges of its elements is $\cl M$.

\begin{theorem}\label{th_ai}
Let $\cl Y\subseteq \cl X$ be a closed $\frak{C}$-invariant subspace and $\cl M\subseteq \cl X$ be an
approximately $\frak{C}$-injective subspace.

(i) \ The space $\cl Y + \cl M$ is closed.

(ii) We have $\Ref_{\frak{C}}(\cl Y + \cl M) = \Ref_{\frak{C}}(\cl
Y) + \cl M$. Thus, if $\cl Y$ is $\frak{C}$-reflexive then $\cl Y +
\cl M$ is $\frak{C}$-reflexive.
\end{theorem}
\begin{proof}
Let $(\phi_n)_{n=1}^{\infty}$ be an associated sequence of idempotents for $\cl M$.

(i) Suppose that $x\in \overline{\cl Y + \cl M}$. Proposition \ref{p_sin} (ii) implies that
$$x\in \overline{\cl Y + \ran\phi_n} = \cl Y + \ran\phi_n.$$
Thus, $x = y_n + m_n$, where $y_n \in \cl Y$ and $m_n\in \ran\phi_n$, and so
$$x - \phi_n(x) = y_n + m_n - \phi_n(y_n) - \phi_n(m_n) = y_n - \phi_n(y_n) \in \cl Y$$ by
the $\frak{C}$-invariance of $\cl Y$.
Let $m$ be a cluster point of $(\phi_n(x))_{n=1}^{\infty}$. Then $x = (x-m) + m\in \cl Y + \cl M$.

(ii) Let $x\in \Ref_{\frak{C}}(\cl Y + \cl M)$.
By Proposition \ref{p_sin} (iii),
$$x\in \Ref\mbox{}_{\frak{C}}(\cl Y + \ran\phi_n) = \Ref\mbox{}_{\frak{C}}\cl Y + \ran\phi_n.$$
Using Proposition \ref{p_intr} (iii), we see as in (i) that
$$x = (x-\phi_n(x)) + \phi_n(x)\in \Ref\mbox{}_{\frak{C}}\cl Y + \ran\phi_n.$$
Taking a cluster point of $(\phi_n(x))_{n=1}^{\infty}$, we conclude that $x\in \Ref\mbox{}_{\frak{C}}\cl Y + \cl M$.
The last two statements are now immediate.
\end{proof}

If $(\cl M_n)_{n=1}^{\infty}$ is a sequence of closed subsets of
$\cl X$, let $\limsup_{n\in \bb{N}}\cl M_n$ be the set of all
cluster points of sequences $(x_n)_{n=1}^{\infty}$, where $x_n \in
\cl M_n$, $n\in \bb{N}$.

\begin{definition}\label{d_ad}
A closed subspace $\cl V\subseteq \cl X$ will be called $\frak{C}$- \emph{decomposable}
if there exists a sequence $(\phi_n)_{n=1}^{\infty} \subseteq \frak{C}$ and a sequence $(\cl W_n)_{n=1}^{\infty}$
of $\frak{C}$-injective subspaces of $\cl X$ such that
\begin{itemize}
\item [(a)] there exists $C > 0$ with $\|\phi_n\|\leq C$, $n\in \bb{N}$;
\item [(b)] $\cl V\subseteq \ran\phi_n + \cl W_n$ for each $n$;
\item [(c)] $\cl W_n\subseteq \cl V$ for each $n$;
\item [(d)] $\limsup_{n\in \bb{N}}\ran\phi_n \subseteq \cl V$.
\end{itemize}
We call the sequence $(\phi_n)_{n=1}^{\infty}$ an
\emph{associated sequences of idempotents}, and $(\cl W_n)_{n=1}^{\infty}$
an \emph{associated sequences of subspaces, for $\cl V$}.
\end{definition}

\begin{lemma}\label{l_ade}
Let $\cl V$ be a $\frak{C}$-decomposable subspace with associated sequences
$(\phi_n)_{n=1}^{\infty}$ and $(\cl W_n)_{n=1}^{\infty}$ of idempotents and subspaces,
respectively. Then
$$\cl V = \cap_{n=1}^{\infty}(\ran\phi_n + \cl W_n) = \limsup\mbox{}_{n\in \bb{N}} \cl W_n + \limsup\mbox{}_{n\in \bb{N}} \ran\phi_n.$$
\end{lemma}
\begin{proof}
By definition, $\cl V \subseteq \cap_{n=1}^{\infty}(\ran\phi_n + \cl W_n)$.
Suppose that $x\in \ran\phi_n + \cl W_n$ for each $n$ and write
$x = x_n + w_n$, where $\phi_n(x_n) = x_n$ and $w_n\in \cl W_n$, $n\in \bb{N}$.
By Proposition \ref{p_intr} (v), $\cl W$ is $\frak{C}$-invariant and hence
$$x - \phi_n(x) = x_n + w_n - \phi_n(x_n) - \phi_n(w_n) = w_n - \phi_n(w_n)\in \cl W_n.$$
Letting $m$ be a cluster point of $(\phi_n(x))_{n=1}^{\infty}$ (such a cluster point exists since
the sequence $(\phi_n)_{n\in \bb{N}}$ is uniformly bounded),
we see that
$$x = (x-m) + m \in \limsup\mbox{}_{n\in \bb{N}} \cl W_n +  \limsup\mbox{}_{n\in \bb{N}} \ran\phi_n.$$

Since $\cl W_n\subseteq \cl V$ for each $n$ and $\cl V$ is closed, we have that
$\limsup_{n\in \bb{N}} \cl W_n \subseteq \cl V$. From condition (d),
$\limsup_{n\in \bb{N}} \cl W_n +  \limsup_{n\in \bb{N}} \ran\phi_n \subseteq \cl V$.
The equalities are established.
\end{proof}

\noindent {\bf Remarks (i) }
Every $\frak{C}$-injective subspace is trivially approximately $\frak{C}$-injective.
Taking $\cl W_n = \{0\}$ in Definition \ref{d_ad} we see, on the other hand, that every
approximately $\frak{C}$-injective subspace is $\frak{C}$-decomposable.

\smallskip

\noindent {\bf (ii) }
It follows from Lemma \ref{l_ade}, Proposition \ref{p_intr} and Proposition \ref{p_sin}
that every $\frak{C}$-decomposable subspace is $\frak{C}$-invariant and $\frak{C}$-reflexive.

\smallskip

\noindent {\bf (iii) } In concrete applications, one often has that the sequence $(\cl W_n)_{n\in \bb{N}}$
in Definition \ref{d_ad} is increasing, while the sequence $(\ran\phi_n)_{n\in \bb{N}}$ is
decreasing. In this case $\limsup_{n\in \bb{N}} \ran\phi_n = \cap_{n=1}^{\infty}\ran\phi_n$
and $\limsup_{n\in \bb{N}} \cl W_n = \overline{\cup_{n=1}^{\infty}\cl W_n}$.

\begin{theorem}\label{th_adr}
Let $\cl Y\subseteq \cl X$ be a closed $\frak{C}$-invariant subspace and $\cl V\subseteq \cl X$ be a
$\frak{C}$-decomposable subspace. Then $\Ref_{\frak{C}}(\cl Y + \cl V) = \overline{\Ref_{\frak{C}}(\cl Y) + \cl V}$.
In particular, if $\cl Y$ is $\frak{C}$-reflexive then $\overline{\cl Y + \cl V}$ is $\frak{C}$-reflexive.
\end{theorem}
\begin{proof}
Let $(\phi_n)$ and $(\cl W_n)$ be associated sequences of idempotents and subspaces for $\cl V$.
By Proposition \ref{p_intr} (i) and (iv), $\overline{\Ref_{\frak{C}}\cl Y + \cl V} \subseteq \Ref_{\frak{C}}(\cl Y + \cl V)$.
Fix $x\in \Ref_{\frak{C}}(\cl Y + \cl V)$.
Letting $\cl M_n = \ran\phi_n$, we have by Proposition \ref{p_sin} (iii) that
$\Ref_{\frak{C}}(\cl Y + \cl M_n + \cl W_n) =  \Ref_{\frak{C}}(\cl Y) + \cl M_n + \cl W_n$.
By Proposition \ref{p_intr} (i) and Theorem \ref{th_ai} (ii), we have that
$$\Ref\mbox{}_{\frak{C}}(\cl Y + \cl V)\subseteq \Ref(\cl Y + \cl W_n + \cl M_n) = \Ref(\cl Y) + \cl W_n + \cl M_n$$
and the arguments in the proof of Theorem \ref{th_ai} show that
$$x-\phi_n(x)\in \Ref\mbox{}_{\frak{C}}(\cl Y) + \cl W_n \subseteq \Ref\mbox{}_{\frak{C}}(\cl Y) + \cl V.$$
Choosing a cluster point $z$ of $(\phi_n(x))_{n=1}^{\infty}$, we have that
$$x = (x-z) + z \in \overline{\Ref\mbox{}_{\frak{C}}(\cl Y) + \cl V} + \limsup\mbox{}_{n\in \bb{N}} \cl M_n \subseteq
\overline{\Ref\mbox{}_{\frak{C}}(\cl Y) + \cl V}.$$
We thus showed that $\Ref_{\frak{C}}(\cl Y + \cl V) \subseteq \overline{\Ref_{\frak{C}}\cl Y + \cl V}$
and hence we have equality. The last statement is now immediate.
\end{proof}

Our last aim in this section is Theorem \ref{th_int} which establishes a useful intersection
property for $\frak{C}$-decomposable subspaces. First, we need a lemma.

\begin{lemma}\label{l_j}
Let $\cl Y\subseteq \cl X$ be a closed $\frak{C}$-invariant subspace
and $\cl V_1,\dots,\cl V_k$ be $\frak{C}$-decomposable subspaces.
Set $\cl V_{i,j} = \cl V_i\cap\dots\cap\cl V_j$ if $i \leq j$, and
$\cl V_{i,j} = \cl X$ if $i > j$. Then, for every $j=1,\dots,k$, we
have
$$\cl V_{1,j} \cap \overline{\cl Y + \cl V_{j+1,k}} \subseteq \overline{\cl Y + \cl V_{1,k}}.$$
\end{lemma}
\begin{proof}
First note that the spaces $\cl V_{i,j}$ are $\frak{C}$-invariant
since $\cl V_1,\dots,\cl V_n$ are $\frak{C}$-invariant. To prove the
statement, use backward induction on $j$. If $j = k$ then the
inclusion is trivial. Suppose that it holds for $j+1$. Let
$(\phi_n)_{n\in \bb{N}}$ be an associated sequence of idempotents,
and $(\cl W_n)_{n\in \bb{N}}$ be an associated sequence of
subspaces, for $\cl V_{j+1}$, and let $\psi_n\in \frak{C}$ be an
idempotent with range $\cl W_n$, $n\in \bb{N}$. Then $\sigma_n
\stackrel{def}{=}\phi_n + \psi_n - \phi_n \psi_n$ has range
$\ran\phi_n + \cl W_n$. Set $\rho_n = \psi_n - \phi_n\psi_n$.

Let $x\in \cl V_{1,j} \cap \overline{\cl Y + \cl V_{j+1,k}}$;
we have that
$$x = \phi_n(x) + \rho_n(x) + \sigma_n^{\perp}(x).$$
Since $x\in \overline{\cl Y + \cl V_{j+1}}$, $\cl V_{j+1}\subseteq \ran\sigma_n$ and $\cl Y$ is $\frak{C}$-invariant,
we have that $\sigma_n^{\perp}(x)\in \cl Y$ for each $n$. On the other hand,
by the $\frak{C}$-invariance of $\cl V_{1,j}$ and the fact that $\ran\rho_n\subseteq \cl W_n$, we have that
$$\rho_n(x)\in \cl V_{1,j}\cap \cl W_n\subseteq \cl V_{1,j}\cap \cl V_{j+1} = \cl V_{1,j+1}.$$
Since $\overline{\cl Y + \cl V_{j+1,k}}$ is invariant under $\rho_n$,
we conclude that
$$\rho_n(x) \in \cl V_{1,j+1} \cap \overline{\cl Y + \cl V_{j+1,k}}
\subseteq
\cl V_{1,j+1} \cap \overline{\cl Y + \cl V_{j+2,k}} \subseteq \overline{\cl Y + \cl V_{1,k}}.$$
Thus, $\rho_n(x) + \sigma_n^{\perp}(x)\in \overline{\cl Y + \cl V_{1,k}}$ for each $n$.

Let $y$ be a cluster point of $(\phi_n(x))_{n\in \bb{N}}$. Then
$$y\in \left(\limsup\mbox{}_{n\in \bb{N}}\ran\phi_n\right)\cap \cl V_{1,j} \subseteq \cl V_{j+1} \cap \cl V_{1,j}
= \cl V_{1,j+1}.$$
By $\frak{C}$-invariance, $y$ also belongs to $\overline{\cl Y + \cl V_{j+1,k}}$ and hence,
by the inductive assumption,
$y\in \overline{\cl Y + \cl V_{1,k}}$.
On the other hand, $x-y$ is a cluster point of $(\rho_n(x) + \sigma_n^{\perp}(x))_{n\in \bb{N}}$ and hence
$x = y + (x-y)\in \overline{\cl Y + \cl V_{1,k}}$.
\end{proof}

\begin{theorem}\label{th_int}
Let $\cl Y\subseteq \cl X$ be a closed $\frak{C}$-invariant subspace
and $\cl V_1,\dots,\cl V_k$ be $\frak{C}$-decomposable subspaces of $\cl X$.
Then
$$\overline{\cl Y + \cap_{j=1}^k\cl V_j} = \cap_{j=1}^k \overline{\cl Y + \cl V_j}.$$
Moreover, if $\cl Y$ is $\frak{C}$-reflexive then
$\overline{\cl Y + \cap_{j=1}^k\cl V_j}$ is $\frak{C}$-reflexive.
\end{theorem}
\begin{proof}
By induction, it suffices to show that if
$\cl V = \cap_{j=1}^{k-1}\cl V_j$ then
\begin{equation}\label{eq_last}
\overline{\cl Y + \cl V}\cap \overline{\cl Y + \cl V_k} = \overline{\cl Y + \cl V\cap\cl V_k}.
\end{equation}
Let $x$ be an element of the left hand side of (\ref{eq_last}).
Let $(\phi_n)_{n\in \bb{N}}$ be an associated sequence of maps,
and $(\cl W_n)_{n\in \bb{N}}$ be an associated sequence of spaces, for $\cl V_k$.
Then, for every $n$ we have that $x = y + \phi_n(x) + y_n$ for some $y\in \cl Y$ and $y_n\in \cl W_n$.
Since $\overline{\cl Y + \cl V}$ is invariant under $\phi_n$, we have
$$x - \phi_n(x) = y + y_n \in \overline{\cl Y + \cl V} \cap (\cl Y + \cl W_n).$$
Suppose that
$y + y_n = \lim_{\alpha} (s_{\alpha} + t_{\alpha})$, where $s_{\alpha}\in \cl Y$ and $t_{\alpha}\in \cl V$.
Writing $\psi_n$ for the idempotent in $\frak{C}$ with range $\cl W_n$, we have
$$y_n = \psi(y_n) = \lim\mbox{}_{\alpha} \psi_n(s_{\alpha} - y) + \psi_n(t_{\alpha}).$$
By $\frak{C}$-invariance,
$\psi_n(t_{\alpha})\in \cl V\cap \cl W_n \subseteq \cl V\cap \cl V_k$ and $\psi_n(s_{\alpha} - y)\in \cl Y$.
It follows that $y_n\in \overline{\cl Y + (\cl V\cap \cl V_k)}$, and so
$y + y_n\in \overline{\cl Y + (\cl V\cap \cl V_k)}$, for every $n$.

Suppose that $\lim_{k\rightarrow\infty}\phi_{n_k}(x) = z$ for some $z$; thus
$z\in \limsup_{n\in \bb{N}} \ran\phi_n \subseteq \cl V_k$.
By Lemma \ref{l_j},
$$z\in \cl V_k\cap \overline{\cl Y + \cl V} \subseteq \overline{\cl Y + \cl V\cap \cl V_k}.$$
Also, $x - z = \lim_{k\rightarrow\infty} (y + y_{n_k}) \in \overline{\cl Y + \cl V\cap \cl V_k}$.
It follows that $x\in \overline{\cl Y + \cl V\cap \cl V_k}$ and (\ref{eq_last}) is established.

Now suppose that $\cl Y$ is $\frak{C}$-reflexive. By Theorem \ref{th_adr},
$\overline{\cl Y + \cl V_j}$ is $\frak{C}$-reflexive for each $j$.
Now Proposition \ref{p_intr} (ii) implies that
$\cap_{j=1}^k\overline{\cl Y + \cl V_j}$, and hence $\overline{\cl Y + \cap_{j=1}^k\cl V_j}$, is $\frak{C}$-reflexive.
\end{proof}

%\noindent {\bf Remark } The results of this section remain true if $\cl X$ is an arbitrary
%topological vector space (and not necessarily a dual Banach space equipped with its weak* topology),
%if some changes are made in Definitions

\section{Sums of masa-bimodules}\label{s_amb}

In this section, we apply the results of Section \ref{s_gf} in the case
where $\cl X = \cl B(H_1,H_2)$ for some Hilbert spaces $H_1$ and $H_2$,
equipped with its canonical weak* topology coming from the identification
of $\cl B(H_1,H_2)$ with the dual space of the space $\cl C_1(H_2,H_1)$ of all trace class operators
from $H_2$ into $H_1$.

We first fix notation. Let $(X,m)$ and $(Y,n)$ be standard measure spaces, that is, the measures $m$ and $n$ are
regular Borel measures with respect to some Borel structures on $X$ and $Y$ arising from
complete metrizable topologies.
Let $H_1 = L^2(X,m)$, $H_2 = L^2(Y,n)$, and $\cl D_1$ (resp. $\cl D_2$)
be the algebra of all multiplication operators on $H_1$
(resp. $H_2$) by functions from $L^{\infty}(X,m)$ (resp. $L^{\infty}(Y,n)$).
It is well-known that $\cl D_1$ and $\cl D_2$ are maximal abelian selfadjoint subalgebras (masas) of
$\cl B(H_1)$ and $\cl B(H_2)$, respectively.
A subspace $\cl U\subseteq \cl B(H_1,H_2)$ will be called a \emph{masa-bimodule} if
$BTA\in \cl U$ whenever $A\in \cl D_1$, $B\in \cl D_2$ and $T\in \cl U$.

We need several facts and notions from the theory of masa-bimodules
\cite{a}, \cite{eks}, \cite{st1}. A subset $E\subseteq X\times Y$ is
called {\it marginally null} if $E\subseteq (X_0\times
Y)\cup(X\times Y_0)$, where $m(X_0) = n(Y_0) = 0$. We call two
subsets $E,F\subseteq X\times Y$ {\it marginally equivalent} (and
write $E\cong F$) if the symmetric difference of $E$ and $F$ is
marginally null. A set $\kappa\subseteq X\times Y$ is called {\it
$\omega$-open} if it is marginally equivalent to a (countable) union
of the form $\cup_{i=1}^{\infty} \alpha_i\times\beta_i$, where
$\alpha_i\subseteq X$ and $\beta_i\subseteq Y$ are measurable, $i\in
\bb{N}$. The complements of $\omega$-open sets are called {\it
$\omega$-closed}. An operator $T\in \cl B(H_1,H_2)$ is said to be
supported on $\kappa$ if $M_{\chi_{\beta}}TM_{\chi_{\alpha}} = 0$
whenever $(\alpha\times\beta)\cap \kappa \simeq \emptyset$. (Here
$M_g$ stands for the operator of multiplication by the function
$g$.) If $\kappa$ is an $\omega$-closed set, let
$$\frak{M}_{\max}(\kappa) = \{T\in \cl B(H_1,H_2) : T \mbox{ is supported on } \kappa\}.$$
The space $\frak{M}_{\max}(\kappa)$ is a reflexive masa-bimodule
and, conversely, every reflexive masa-bimodule is of this form, for
some $\omega$-closed set $\kappa$ \cite{eks}. Given a weak* closed
masa-bimodule $\cl U$, its \emph{support} is the $\omega$-closed set
$\kappa$ such that $\Ref\cl U = \frak{M}_{\max}(\kappa)$. Given an
$\omega$-closed $\kappa\subseteq X\times Y$, there exists a smallest
(with respect to inclusion) weak* closed masa-bimodule $\cl U$ with
support $\kappa$ \cite{a}, \cite{st1}; we denote this minimal
masa-bimodule by $\frak{M}_{\min}(\kappa)$. We will often use the
fact that if $\kappa_1$ and $\kappa_2$ are $\omega$-closed sets with
$\kappa_1\subseteq \kappa_2$ then
$\frak{M}_{\max}(\kappa_1)\subseteq \frak{M}_{\max}(\kappa_2)$ (this
follows from the definition of $\frak{M}_{\max}(\kappa)$) and
$\frak{M}_{\min}(\kappa_1)\subseteq \frak{M}_{\min}(\kappa_2)$ (this
follows from \cite[Theorem 3.3]{st1}).

Recall that the projective tensor product $\Gamma(X,Y)
\stackrel{def}{=} L^2(X,m) \hat{\otimes } L^2(Y,n)$, whose norm will
be denoted by $\|\cdot\|_{\Gamma}$, can be canonically identified
with the predual of $\cl B(H_1,H_2)$. Each element $h\in
\Gamma(X,Y)$ can be associated with a series (convergent with
respect to $\|\cdot\|_{\Gamma}$) $h \sim \sum_{i=1}^{\infty}
f_i\otimes g_i$, where $\sum_{i=1}^{\infty} \|f_i\|_2^2 < \infty$
and $\sum_{i=1}^{\infty} \|g_i\|_2^2 < \infty$; we have that
$\langle T,h\rangle = \sum_{i=1}^{\infty} (Tf_i,\overline{g_i})$. It
follows \cite{a} that $h$ can be identified with a complex function
on $X\times Y$, defined up to a marginally null set, and given by
$h(x,y) = \sum_{i=1}^{\infty} f_i(x)g_i(y)$. A function $\nph \in
L^{\infty}(X\times Y,m\times n)$ is called a \emph{Schur multiplier}
if $\nph h$ is equivalent (with respect to the product measure) to a
function from $\Gamma(X,Y)$ for every $h\in \Gamma(X,Y)$ (this
definition is equivalent to other definitions used in the
literature, see \cite{peller}). The Closed Graph Theorem implies
that pointwise multiplication by a Schur multiplier $\nph$ is
bounded, and by taking its dual, we obtain a bounded map $S_{\nph}$
on $\cl B(H_1,H_2)$, which we call a \emph{Schur map}. It is
standard to verify that $S_{\nph}$ is a masa-bimodule map in the
sense that $S_{\nph}(BTA) = BS_{\nph}(T)A$, for all $T\in \cl
B(H_1,H_2)$, $A\in \cl D_1$ and $B\in \cl D_2$. Indeed, we have the
following well-known fact, which follows from results of U. Haagerup
\cite{haag} and R. R. Smith \cite{smith}.

\begin{theorem}\label{th_schur}
Let $\Phi \in \cl B(\cl B(H_1,H_2))$. The following are equivalent:

(i) \ \ $\Phi$ is a bounded weak* continuous masa-bimodule map;

(ii) \ $\Phi$ is a completely bounded weak* continuous masa-bimodule map;

(iii) $\Phi = S_{\nph}$ for some Schur multiplier $\nph$;

(iv) there exists a bounded column operator $(A_i)_{i\in \bb{N}}$ with entries in $\cl D_1$ and
a bounded row operator $(B_i)_{i\in \bb{N}}$ with entries in $\cl D_2$ such that
$\Phi(T) = \sum_{i=1}^{\infty} B_iTA_i$, $T\in \cl B(H_1,H_2)$ (the series being weak* convergent).
\end{theorem}

The Schur multiplier $\nph$ associated with a Schur map $\Phi$ in
(iii) is uniquely determined by $\Phi$ and called its symbol
\cite{kp}. We let $\frak{I}$ be the set of all idempotent Schur maps
(which we will call \emph{Schur idempotents}). It was shown in
\cite{kp} that the symbols of the maps from $\frak{I}$ are
characteristic functions of subsets of $X\times Y$ that are both
$\omega$-closed and $\omega$-open. If $\nph$ and $\psi$ are Schur
multipliers, one easily checks that $S_{\nph} S_{\psi} =
S_{\nph\psi}$; it follows that $\frak{I}$ is a Boolean algebra of
pairwise commuting idempotents in the sense of Section \ref{s_gf}.

\medskip

\noindent {\bf Remark } Recall that an operator space $\cl U
\subseteq \cl B(H_1,H_2)$ is called \emph{injective} if for every
pair of operator spaces $\cl U_1\subseteq \cl U_2$, every completely
bounded map $\Phi_1 : \cl U_1\rightarrow \cl U$ has a completely
bounded extension $\Phi_2$ to $\cl U_2$ with $\|\Phi_1\|_{cb} =
\|\Phi_2\|_{cb}$ \cite{paulsen}. It follows from paragraph 1.6.1 of
\cite{blm} that every approximately $\frak{I}$-injective
masa-bimodule is the range of a completely bounded (not necessarily
weak* continuous) idempotent. Hence, Arveson's Extension Theorem
implies that every approximately $\frak{I}$-injective masa-bimodule
has an extension property for completely bounded maps, not
necessarily with preservation of the completely bounded norm.
%Hence, approximately $\frak{I}$-injective masa-bimodules
%(and, indeed, the $\frak{I}$-injective ones) are not necessarily injective in the sense of Operator Space Theory
%({\bf example here?}).

\begin{proposition}\label{p_mb}
Let $\cl U\subseteq \cl B(H_2,H_1)$ be a weak* closed subspace. The following are equivalent:

(i) \ \ $\cl U$ is a masa-bimodule;

(ii) \ $\cl U$ is invariant under all Schur maps;

(iii) $\cl U$ is $\frak{I}$-invariant.
\end{proposition}
\begin{proof}
(i)$\Rightarrow$(ii) Since $\cl U$ is a masa-bimodule, $ATB\in \cl U$ whenever $A\in \cl D_1$,
$B\in \cl D_2$ and $T\in \cl U$. Since $\cl U$ is weak* closed, condition (iv) of
Theorem \ref{th_schur} implies that $\cl U$ is invariant under all Schur maps.

(ii)$\Rightarrow$(iii) is trivial.

(iii)$\Rightarrow$(i) Let $E$ (resp. $F$) be a projection in $\cl D_1$ (resp. $\cl D_2$).
Then the map given by $\Phi(T) = FTE$, $T\in \cl B(H_1,H_2)$, clearly belongs to $\frak{I}$.
Thus, $F\cl U E\subseteq \cl U$. Since every von Neumann algebra is generated by its projections
in the norm topology,
we conclude that $B\cl U A \subseteq \cl U$ for all $A\in \cl D_1$ and $B\in \cl D_2$.
\end{proof}

\begin{proposition}\label{l_iref}
(i) \ Every $\frak{I}$-injective masa-bimodule is reflexive and is
generated as a weak* closed subspace by the rank one operators it
contains.

(ii) A weak* closed masa-bimodule is reflexive if and only if it is $\frak{I}$-reflexive.
\end{proposition}
\begin{proof}
(i)
Let $\cl U$ be a $\frak{I}$-injective masa-bimodule.
Then there exists a subset $\kappa\subseteq X\times Y$ that is both $\omega$-closed and $\omega$-open such that
$\cl U = \ran S_{\chi_{\kappa}}$.
It is easy to see (or, alternatively, it follows from Proposition \ref{p_intr} (iv))
that $\cl U$ is weak* closed.
By \cite[Proposition 12]{kp}, $\cl U\subseteq \frak{M}_{\max}(\kappa)$. Applying the same argument
to $S_{\chi_{\kappa}}^{\perp} = S_{\chi_{\kappa^c}}$, we obtain
$\ran S_{\chi_{\kappa}}^{\perp}\subseteq \frak{M}_{\max}(\kappa^c)$. It follows that
\begin{eqnarray*}
\cl B(H_1,H_2) = \ran S_{\chi_{\kappa}} +
\ran S_{\chi_{\kappa}}^{\perp} \subseteq \frak{M}_{\max}(\kappa) + \frak{M}_{\max}(\kappa^c)
\subseteq \cl B(H_1,H_2)
\end{eqnarray*}
and hence equality holds throughout. Since the sums are direct, a
simple linear algebra argument shows that $\ran S_{\chi_{\kappa}} = \frak{M}_{\max}(\kappa)$
and hence $\cl U$ is reflexive.

Assume that $\kappa = \cup_{i=1}^{\infty}\alpha_i\times\beta_i$,
where $\alpha_i\subseteq X$, $\beta_i\subseteq Y$ are measurable. By
\cite[Lemma 3.4]{eks}, for each $N\in\bb{N}$, there exist sets
$X_N\subseteq X$ and $Y_N\subseteq Y$ such that $m(X\setminus X_N) <
\frac{1}{N}$ and $n(Y\setminus Y_N) < \frac{1}{N}$ and $\kappa\cap
(X_N\times Y_N)$ is contained in the union of finitely many of the
sets $\alpha_i\times\beta_i$. Since a finite union of Borel
rectangles is the finite union of disjoint Borel rectangles and
$\frak{M}_{\max}(\alpha\times\beta) = \cl
B(M_{\chi_{\alpha}}H_1,M_{\chi_{\beta}}H_2)$, we have that
$\frak{M}_{\max}((X_N\times Y_N) \cap \kappa)$ is the weak* closure
of the linear span of its rank one operators. Now let $T\in
\frak{M}_{\max}(\kappa)$ be arbitrary. Then $T =
$w$^*$-$\lim_{N\rightarrow \infty} F_N T E_N$, where $E_N$ (resp.
$F_N$) is the projection of multiplication by $\chi_{X_N}$ (resp.
$\chi_{Y_N}$). Moreover, $F_NTE_N\in \frak{M}_{\max}((X_N\times Y_N)
\cap \kappa)$, and the claim follows.

(ii)
Let $\cl U \subseteq \cl B(H_1,H_2)$ be a weak* closed masa-bimodule.
By \cite{st1},
$\Ref\cl U$ consists of all the operators $T\in \cl B(H_1,H_2)$ with $FTE = 0$ whenever
$E\in \cl D_1$ and $F\in \cl D_2$ are projections with $F\cl U E = \{0\}$.
We claim that
\begin{equation}\label{eq_eqre}
\Ref\cl U = \Ref\mbox{}_{\frak{I}}\cl U.
\end{equation}
Suppose that $T\in \Ref_{\frak{I}}\cl U$ and let
$E\in \cl D_1$ and $F\in \cl D_2$ be projections with $F\cl U E = \{0\}$.
The mapping $\Phi$ on $\cl B(H_1,H_2)$ given by $\Phi(X) = FXE$ clearly belongs to $\frak{I}$
and annihilates $\cl U$.
By the definition of $\Ref_{\frak{I}}$, we have that $\Phi(T) = 0$, that is, $FTE = 0$.
Thus, $T\in \Ref\cl U$.

Conversely, suppose that $T\in \Ref\cl U$ and that $\Phi \in \frak{I}$ annihilates $\cl U$.
By (i), $\ker \Phi = \ran(\Phi^{\perp})$ is reflexive, and hence $\Ref\cl U\subseteq \ker\Phi$.
Thus, $\Phi(T) = 0$ and so $T\in \Ref_{\frak{I}}\cl U$.
Equality (\ref{eq_eqre}) is now established and the conclusion is immediate from it.
\end{proof}

Theorems \ref{th_ai} and \ref{th_int} now yield the following corollaries.

\begin{corollary}\label{c_corm}
Let $\cl V$ be an approximately $\frak{J}$-injective masa-bimodule and $\cl U$ be
a weak* closed masa-bimodule. Then $\cl U + \cl V$ is weak* closed and
$\Ref(\cl U + \cl V) = \Ref(\cl U) + \cl V$. In particular, if
$\cl U$ is reflexive then $\cl U + \cl V$ is reflexive.
\end{corollary}

\begin{corollary}\label{c_corad}
Let $\cl V_1,\dots,\cl V_k$ be $\frak{J}$-decomposable masa-bimodules and $\cl U$ be
a weak* closed masa-bimodule. Then
$$\overline{\cl U + \cap_{j=1}^k \cl V_j}^{w^*} = \cap_{j=1}^k \overline{\cl U + \cl V_j}^{w^*}.$$
Moreover, if $\cl U$ is reflexive then $\overline{\cl U + \cap_{j=1}^k \cl V_j}^{w^*}$ is a reflexive
masa-bimodule.
In particular, if $\cl U$ is reflexive then $\overline{\cl U + \cl V_1}^{w^*}$ is reflexive.
\end{corollary}

The example that follow show some particular instances where
Corollaries \ref{c_corm} and \ref{c_corad} can be applied.

\noindent {\bf Examples (i) } We recall \cite{houston} that a weak*
closed masa-bimodule $\cl M$ is called \emph{ternary} if $\cl M$ is
a TRO, that is, if $TS^*R\in \cl M$ whenever $T,S,R\in \cl M$ (see
\cite{blm}). The class of ternary masa-bimodules includes all von
Neumann algebras with abelian commutant. We claim that every weak*
closed ternary masa-bimodule $\cl M$ is approximately
$\frak{J}$-injective. To see this, note first that one may assume
that $\cl M H_1$ is dense in $H_2$ and $\cl M^*H_2$ is dense in
$H_1$, for otherwise we can replace the spaces $H_1$ and $H_2$ by
$H_2^0 = \overline{\cl M H_1}$ and $H_1^0 = \overline{\cl M^* H_2}$,
respectively. Set $\cl C_1 = (\cl M^*\cl M)'$ and $\cl C_2 = (\cl
M\cl M^*)'$. By \cite{kt}, there exists a strongly continuous
Boolean algebra isomorphism $\theta : \Proj(\cl C_1) \rightarrow
\Proj(\cl C_2)$ (where by $\Proj(\cl C)$ we denote the set of all
orthogonal projections in $\cl C$) such that
$$\cl M = \{T\in \cl B(H,K) : TP = \theta(P)T, \ P\in \Proj(\cl C_1)\}.$$
Let $(P_k)_{k=1}^{\infty}$ be a strongly dense sequence in
$\Proj(\cl C_1)$. Let $E_1,\dots,E_{m_n}$ be the atoms of the von
Neumann algebra generated by $P_1,\dots,P_n$, $F_j = \theta(E_j)$,
and $\cl E_n\in \frak{J}$ be given by $\cl E_n(T) = \sum_j F_j T
E_j$. Then $\|\cl E_n\|\leq 1$ for all $n$, $\ran\cl
E_{n+1}\subseteq \ran \cl E_n$ and $\cl M = \cap_{n=1}^{\infty} \ran
\cl E_n$. Thus, $\cl M$ is approximately $\frak{I}$-injective.

It follows that the class of approximately $\frak{I}$-injective
masa-bimodules is strictly larger than that of $\frak{I}$-injective
ones. Indeed, by the previous paragraph, a continuous masa is
approximately $\frak{I}$-injective, but it is not
$\frak{I}$-injective by a well-known result of Arveson's
\cite{a_ajm}.

In Section \ref{s_syn} we will give an example of an approximately
$\frak{I}$-injective masa-bimodule for which the uniform bound on
the norms of the corresponding idempotents cannot be chosen to be 1.

\smallskip

\noindent {\bf (ii) } Recall that a nest is a totally ordered
strongly closed set of projections on a Hilbert space, and that a
nest algebra is the algebra of all operators leaving a given nest
invariant. A nest algebra bimodule is a subspace $\cl V$ for which
there exists nest algebras $\cl A$ and $\cl B$ with $\cl B \cl V \cl
A\subseteq \cl V$. We claim that every weak* closed nest algebra
bimodule $\cl V$ is $\frak{I}$-decomposable. To see this, write
\cite{ep}
$$\cl V = \{S\in \cl B(H_1,H_2) : SN = \nph(N)SN, \ N\in \cl N_1\},$$
for some nest $\cl N_1\subseteq \cl B(H_1)$ and an increasing
$\vee$-preserving map $\nph : \cl N_1\rightarrow \cl N_2$ ($\cl N_2$
being a nest on $H_2$). For every finite family $\cl F$ consisting
of the elements $0 = N_1 < N_2 < \dots < N_k = I$ of $\cl N$, let
$\cl E_{\cl F}$ be given by $\cl E_{\cl F}(T) = \sum_{i=1}^k
(\nph(N_{i+1}) - \nph(N_i))T(N_{i+1} - N_i)$ and let $\cl M_{\cl F}$
be the range of $\cl E_{\cl F}$. Let
$$\cl W_{\cl F} = \sum_{i < j} (\nph(N_{i+1}) - \nph(N_i))T(N_{j+1}
- N_j).$$

Choose a (countable) strongly dense subset $\{N_i\}_{i\in \bb{N}}$
of $\cl N_1$ such that the set $\{\nph(N_i)\}_{i\in \bb{N}}$ is
dense in $\cl N_2$, and let $\cl P_n = \{0,N_1,N_2,\dots,N_n,I\}$.
The conditions of Definition \ref{d_ad} are now readily verified for
the sequences $(\cl E_{\cl P_n})_{n\in \bb{N}}$ and $(\cl W_{\cl
P_n})_{n\in \bb{N}}$.

It follows from the previous two paragraphs
that the Volterra nest algebra $\cl A$ acting on $L^2(0,1)$
is an $\frak{I}$-decomposable masa-bimodule.
However, it is not approximately $\frak{I}$-injective. To see this, assume the converse and
note that, by Theorem \ref{th_decai} below, we have a direct sum decomposition
$\cl A = \cl A_1 + \cl A_2$ such that $\cl A_1$ is an injective masa-bimodule and
$\cl A\cap \cl K\subseteq \cl A_1$. However, $\cl A\cap \cl K$ is weak* dense in $\cl A$
(see, e.g. \cite{dav}) and it would follow that $\cl A$ is injective.
Hence, the function $\chi_{\Delta}$ where $\Delta = \{(x,y)\in [0,1]\times [0,1] : x \leq y\}$
would be a Schur multiplier on $\cl B(L^2(0,1))$, equivalently, the
transformer of triangular truncation would be bounded, a contradiction.

\medskip

The above examples and Corollaries \ref{c_corm} and \ref{c_corad} have the following consequences.

\begin{corollary}\label{c_tros}
Let $\cl U$ be a weak* closed masa-bimodule and $\cl M$ be a weak* closed ternary masa-bimodule.
Then $\cl U + \cl M$ is weak* closed. If $\cl U$ is reflexive then $\cl U + \cl M$ is reflexive.
\end{corollary}

For the next corollary, we recall that a masa-bimodule of finite width is, by definition,
the intersection of finitely many nest algebra bimodules.

\begin{corollary}\label{c_fw}
Let $\cl U$ be a reflexive masa-bimodule and $\cl V$ be a masa-bimodule of finite width. Then
$\overline{\cl U + \cl V}^{w^*}$ is reflexive. In particular, if $\cl W$ is a nest algebra bimodule
then $\overline{\cl U + \cl W}^{w^*}$ is reflexive.
\end{corollary}

\section{Connections with operator synthesis}\label{s_syn}

Let, as in Section \ref{s_amb}, $(X,m)$ and $(Y,n)$ be standard
measure spaces, $H_1 = L^2(X,m)$, $H_2 = L^2(Y,n)$ and $\cl D_1$
(resp. $\cl D_2$) be the multiplication masa of $L^{\infty}(X,m)$
(resp. $L^{\infty}(Y,n)$). We will denote by $\cl K$ (resp. $\cl
C_2$) the ideal of all compact (resp. Hilbert-Schmidt) operators
from $H_1$ into $H_2$. An $\omega$-closed set $\kappa\subseteq
X\times Y$ is called \emph{operator synthetic} \cite{a}, \cite{st1}
if $\frak{M}_{\min}(\kappa) = \frak{M}_{\max}(\kappa)$. A weak*
closed masa-bimodule $\cl U$ will be called \emph{synthetic} if the
(unique up to marginal equivalence) $\omega$-closed subset $\kappa$
such that $\Ref\cl U = \frak{M}_{\max}(\kappa)$ is operator
synthetic.

For an $\omega$-closed set $\kappa\subseteq X\times Y$, we let
$\Psi_{\max}(\kappa) = \frak{M}_{\max}(\kappa)_{\perp}$ and
$\Psi_{\min}(\kappa) = \frak{M}_{\min}(\kappa)_{\perp}$. We have
that $\Psi_{\max}(\kappa)$ and $\Psi_{\min}(\kappa)$ are
$\|\cdot\|_{\Gamma}$-closed subspaces of $\Gamma(X,Y)$, invariant
under pointwise multiplication by Schur multipliers. We say that a
function $h\in \Gamma(X,Y)$ vanishes on a subset $\kappa\subseteq
X\times Y$ (and write \lq\lq $h = 0$ on $\kappa$'') if
$h\chi_{\kappa}(x,y) = 0$ for marginally almost all $(x,y)$. We have
that \cite{st1},
$$\Psi_{\max}(\kappa) = \overline{\{h\in \Gamma(X,Y) : h = 0 \mbox{ on an } \omega\mbox{-open set containing } \kappa\}}^{\|\cdot\|_{\Gamma}}$$
and
$$\Psi_{\min}(\kappa) = \{h\in \Gamma(X,Y) : h = 0 \mbox{ on } \kappa\}.$$
By duality, a subset $\kappa\subseteq X\times Y$ is operator
synthetic if and only if $\Psi_{\max}(\kappa) =
\Psi_{\min}(\kappa)$. The set $\kappa$ is called \emph{strong
operator Ditkin} \cite{st1} if there exists a sequence $(w_n)_{n\in
\bb{N}}$ of Schur multipliers, such that $w_n$ vanishes on an
$\omega$-open set containing $\kappa$, $n\in \bb{N}$, and $\|h - w_n
h\|_{\Gamma}\rightarrow_{n\rightarrow\infty} 0$ for every $h\in
\Psi_{\min}(\kappa)$.

The connection between Question \ref{q_rq} and the problem for the
union of operator synthetic sets is summarised in the following
proposition.

\begin{proposition}\label{p_cws}
Suppose that $\lambda\subseteq X\times Y$ is an operator synthetic
$\omega$-closed set such that $\overline{\cl U +
\frak{M}_{\max}(\lambda)}^{w^*}$ is reflexive whenever $\cl U$ is a
reflexive masa-bimodule. Then $\kappa\cup \lambda$ satisfies
operator synthesis whenever $\kappa$ does so.
\end{proposition}
\begin{proof}
It is easy to see that for every $\omega$-closed set $\kappa$, the
support of $\frak{M}_{\max}(\kappa) + \frak{M}_{\max}(\lambda)$ is
$\kappa\cup \lambda$. If $\kappa$ satisfies operator synthesis then,
using the fact that $\frak{M}_{\min}$ is monotone, we obtain
\begin{eqnarray*}
\frak{M}_{\max}(\kappa\cup\lambda) & = & \overline{\frak{M}_{\max}(\kappa) + \frak{M}_{\max}(\lambda)}^{w^*} =
\overline{\frak{M}_{\min}(\kappa) + \frak{M}_{\min}(\lambda)}^{w^*}\\ & \subseteq & \frak{M}_{\min}(\kappa\cup\lambda)
\subseteq \frak{M}_{\max}(\kappa\cup\lambda)
\end{eqnarray*}
and hence equality holds throughout. In particular,
$\frak{M}_{\min}(\kappa\cup\lambda) = \frak{M}_{\max}(\kappa\cup\lambda)$, that is, $\kappa\cup\lambda$ is operator synthetic.
\end{proof}

We now discuss some consequences of Proposition \ref{p_cws} and the results from Section \ref{s_amb}.

\medskip

\noindent {\bf Sets of finite width. }
An $\omega$-closed subset $\kappa$ is the support of a nest algebra bimodule if and only if it is of the form
$$\kappa = \{(x,y)\in X\times Y : f(x) \leq g(y)\},$$ for some measurable functions $f : X\rightarrow \bb{R}$ and
$g : Y\rightarrow \bb{R}$; see \cite{t_jlms} (such sets will be
called \emph{nest sets}). It follows that the supports of
masa-bimodules of finite width are precisely the sets of solutions
of systems of inequalities of the form $f_i(x)\leq g_i(y)$, $i =
1,\dots,k$, for some measurable functions $f_i : X\rightarrow
\bb{R}$ and $g_i : Y\rightarrow \bb{R}$, $i = 1,\dots,k$ (call such
sets \emph{of finite width}). It was shown in \cite{st1} and
\cite{t_jlms} that sets of finite width are operator synthetic.
Thus, Proposition \ref{p_cws} and Corollary \ref{c_fw} give the
following extensions of this result.

\begin{corollary}\label{c_ufd}
The union of an operator synthetic set and a set of finite width is operator synthetic.
\end{corollary}

Let $G$ be a second countable locally compact group and $\omega : G\rightarrow \bb{R}^+$ be a continuous group homomorphism
(where $\bb{R}^+$ is the multiplicative group of positive reals).
We assume that the Fourier algebra $A(G)$ has a (perhaps unbounded) approximate identity
(this assumption is satisfied by all amenable groups and is needed for the application of the results from \cite{lt}).
For a subset $E\subseteq G$, write $E^* = \{(s,t)\in G\times G : st^{-1}\in E\}$.
For each $t > 0$, let
$$E_{\omega}^t = \{x\in G : \omega(x) \leq t\}.$$
Then
$$(E_{\omega}^t)^* = \{(x,y) \in G\times G : \omega(x) \leq t\omega(y)\}$$
and hence $(E_{\omega}^t)^*$ is a nest set.
The intersections of the form
$$E = E_{\omega_1}^{t_1}\cap\dots\cap E_{\omega_k}^{t_k}$$
are Harmonic Analysis versions of sets of finite width; they have
the property that corresponding set $E^*$ is a set of finite width.
By \cite{lt}, a closed set $E$ is synthetic if and only if the set
$E^*$ is operator synthetic (where $G$ is equipped with left Haar
measure). Corollary \ref{c_ufd} thus has the following immediate
consequence:

\begin{corollary}\label{c_ha}
Let $F\subseteq G$ be a closed set satisfying spectral synthesis.
Then $F\cup (E_{\omega_1}^{t_1}\cap\dots\cap E_{\omega_k}^{t_k})$ satisfies spectral synthesis.
\end{corollary}

\noindent {\bf Ternary sets. }
An $\omega$-closed subset $\kappa\subseteq X\times Y$ is the support of a
ternary masa-bimodule if and only if it is of the form
$$\kappa = \{(x,y)\in X\times Y : f(x) = g(y)\},$$ for some measurable functions $f : X\rightarrow \bb{R}$ and
$g : Y\rightarrow \bb{R}$; see \cite{vs_ter} and \cite{kt} (we call such sets \emph{ternary}).
Corollary \ref{c_tros} recovers (with a different proof) the following fact, which follows from
\cite[Theorem 7.1]{st1} and \cite[Proposition 5.1]{lt}.

\begin{corollary}\label{c_ust}
The union of an operator synthetic set and a ternary set is operator synthetic.
\end{corollary}

\medskip

If $\kappa$ is the support of an $\frak{I}$-injective masa-bimodule,
then Proposition \ref{l_iref} (i) shows that $\kappa$ is operator
synthetic. Indeed, all Hilbert-Schmidt operators in
$\frak{M}_{\max}(\kappa)$ belong to $\frak{M}_{\min}(\kappa)$ by
\cite{a}, and it follows that $\frak{M}_{\max}(\kappa) =
\frak{M}_{\min}(\kappa)$.

We do not know whether the support of a $\frak{I}$-decomposable
masa-bimodule is necessarily operator synthetic. In the next theorem
we show that whenever it is, it is as a matter of fact strong
operator Ditkin.

\begin{theorem}\label{th_uadm}
Let $\kappa$ be the support of an $\frak{I}$-decomposable masa
bimodule. Then $\kappa$ is operator synthetic if and only if it is
strong operator Ditkin.
\end{theorem}
\begin{proof}
For a subset $E\subseteq X\times Y$, set
$$\gamma(E) = \inf\{m(\alpha) + n(\beta) : E\subseteq (\alpha\times Y)\cup(X\times \beta), \ \alpha,\beta \mbox{ measurable}\}$$
(see R. Haydon and V.S. Shulman's paper \cite{hayshul} where this quantity was defined).

Let $\cl V$ be an $\frak{I}$-decomposable masa-bimodule,
$(\Phi_n)_{n=1}^{\infty}$ be a sequence of elements of $\frak{I}$,
$C > 0$ be a constant with $\|\Phi_n\|\leq C$, $n\in \bb{N}$, and
$(\cl W_n)_{n=1}^{\infty}$ be a sequence of $\frak{I}$-injective
masa-bimodules such that the conditions of Definition \ref{d_ad} are
satisfied. Let $\kappa_n\subseteq X\times Y$ and $\sigma_n\subseteq
X\times Y$ be $\omega$-closed sets with $\ran\Phi_n =
\frak{M}_{\max}(\kappa_n)$ and $\cl W_n =
\frak{M}_{\max}(\sigma_n)$; by \cite{kp}, $\kappa_n$ and $\sigma_n$
are also $\omega$-open. Note that $\Phi_n = S_{\chi_{\kappa_n}}$,
$n\in \bb{N}$.

Let $\kappa\subseteq X\times Y$ be the support of $\cl V$.
By Proposition \ref{p_sin} (i) and (iii), $\ran\Phi_n + \cl W_n$ is reflexive; since its support is
easily seen to be equal to $\sigma_n \cup \kappa_n$, we have that
$\ran\Phi_n + \cl W_n = \frak{M}_{\max}(\sigma_n \cup \kappa_n)$.
Conditions (b) and (c) of Definition \ref{d_ad} imply that, up to a marginally null set,
$$\sigma_n\subseteq \kappa\subseteq \sigma_n \cup \kappa_n, \ \ \ n\in \bb{N}.$$

We claim (without the assumption that $\kappa$ is operator
synthetic) that
\begin{equation}\label{eq_con0}
\Psi_{\max}(\kappa) = \{h\in \Psi_{\min}(\kappa) : \|\chi_{\kappa_n}h\|_{\Gamma}\rightarrow 0\}.
\end{equation}
If $\|\chi_{\kappa_n}h\|_{\Gamma}\rightarrow 0$ for some $h\in \Psi_{\min}(\kappa)$
then $h = \lim_{n\rightarrow\infty} \chi_{\kappa_n^c}h$
and the function $\chi_{\kappa_n^c}h$ vanishes on $\kappa\cup \kappa_n = \sigma_n\cup \kappa_n$,
an $\omega$-open neighbourhood of $\kappa$.
This shows that $h\in \Psi_{\max}(\kappa)$.

Conversely, assume that $h\in \Psi_{\max}(\kappa)$. Given $\epsilon > 0$, there exists
an $\omega$-open set $E$ containing $\kappa$
and an element $h_0\in \Gamma(X,Y)$ vanishing on $E$ such that $\|h - h_0\|_{\Gamma} < \frac{\epsilon}{2C}$.

The sets $E^c\cap \kappa_n$, $n\in \bb{N}$, are $\omega$-closed.
Suppose that $T_n\in \frak{M}_{\max}(E^c\cap \kappa_n)$,
$\|T_n\|\leq 1$, $n\in \bb{N}$, and that $T_n\rightarrow T$ in the
weak operator topology. Since $T_n\in \ran\Phi_n$, we have, by
condition (d) of Definition \ref{d_ad}, that $T\in \cl V =
\frak{M}_{\max}(\kappa)$. On the other hand, $T$ clearly belongs to
$\frak{M}_{\max}(E^c)$ since the latter space is weakly closed and
contains all operators $T_n$, $n\in \bb{N}$. It follows that $T$ is
supported on $E^c\cap \kappa = \emptyset$, and hence $T = 0$. It
follows from \cite[Proposition 3.5]{ht} that $\gamma(E^c\cap
\kappa_n)\rightarrow 0$. Hence we can choose measurable subsets
$\alpha_n\subseteq X$ and $\beta_n\subseteq Y$ such that
$$E^c\cap\kappa_n\subseteq (\alpha_n\times Y)\cup(X\times \beta_n) \ \mbox{ and } \
\lim_{n\rightarrow\infty}m(\alpha_n) = \lim_{n\rightarrow\infty}
n(\beta_n) = 0.$$ Set $E_n = \alpha_n\times Y$ and $F_n =
\alpha_n^c\times \beta_n$. We claim that
\begin{equation}\label{eq_less}
\|\chi_{\kappa_n}h_0\|_{\Gamma} \leq C \|\chi_{E_n}h_0\|_{\Gamma} + C \|\chi_{F_n}h_0\|_{\Gamma}.
\end{equation}
To see this, note that
\begin{eqnarray*}
\|\chi_{\kappa_n} h_0\|_{\Gamma} & = & \sup\{|\langle \chi_{\kappa_n\cap E^c} h_0, T\rangle| : \|T\|\leq 1\}\\
& = &
\sup\{|\langle \chi_{\kappa_n\cap E^c}\chi_{E_n\cup F_n} h_0, T\rangle| : \|T\|\leq 1\}\\
& = & \sup\{|\langle \chi_{\kappa_n} h_0, S_{\chi_{E_n\cup F_n}}(T)\rangle| : \|T\|\leq 1\}\\
& = & \sup\{|\langle h_0, S_{\chi_{\kappa_n} \chi_{E_n\cup F_n}}(T)\rangle| : \|T\|\leq 1\}\\
& = & \sup\{|\langle h_0, S_{\chi_{E_n\cup F_n}}(S_{\chi_{\kappa_n}}(T))\rangle| : \|T\|\leq 1\}\\
& \leq & C\sup\{|\langle h_0, S_{\chi_{E_n\cup F_n} }(T)\rangle| : \|T\|\leq 1\}\\
& \leq & C\sup\{|\langle h_0, S_{\chi_{E_n}}(T)\rangle| : \|T\|\leq 1\}\\
& + & C\sup\{|\langle h_0, S_{\chi_{F_n}}(T)\rangle| : \|T\|\leq 1\}\\
& = & C\sup\{|\langle \chi_{E_n}h_0, T\rangle| : \|T\|\leq 1\}\\
& + & C\sup\{|\langle \chi_{F_n}h_0, T\rangle| : \|T\|\leq 1\}\\
& = & C \|\chi_{E_n}h_0\|_{\Gamma} + C \|\chi_{F_n}h_0\|_{\Gamma}.
\end{eqnarray*}
Thus, (\ref{eq_less}) is established.

We claim that
\begin{equation}\label{eq_limi}
\lim_{n\rightarrow\infty} \|\chi_{E_n}h_0\|_{\Gamma} = \lim_{n\rightarrow\infty} \|\chi_{F_n}h_0\|_{\Gamma} = 0.
\end{equation}
To see this, write $h_0 = \sum_{i=1}^{\infty} f_i\otimes g_i\in \Gamma(X,Y)$ with
$\sum_{i=1}^{\infty} \|f_i\|_2^2 < \infty$ and $\sum_{i=1}^{\infty} \|g_i\|_2^2 < \infty$.
Observe that
$$\int_X \left(\sum_{i=1}^{\infty} |f_i|^2\right)dm = \sum_{i=1}^{\infty} \|f_i\|_2^2 < \infty,$$
that is, $\sum_{i=1}^{\infty} |f_i|^2\in L^1(X,m)$.
Since $\chi_{E_n} h = \sum_{i=1}^{\infty} (\chi_{\alpha_n}f_i) \otimes g_i$, we have that
\begin{eqnarray*}
\|\chi_{E_n}h_0\|_{\Gamma}^2 & \leq & \sum_{i=1}^{\infty}\|\chi_{\alpha_n}f_i\|_2^2 \sum_{i=1}^{\infty}\|g_i\|_2^2\\
& = &  \left(\sum_{i=1}^{\infty}\|g_i\|_2^2\right) \int_{\alpha_n}
\left(\sum_{i=1}^{\infty} |f_i|^2\right)dm
\longrightarrow_{n\rightarrow\infty} 0
\end{eqnarray*}
since $m(\alpha_n)\rightarrow 0$.
Similarly we show that $\|\chi_{F_n}h_0\|_{\Gamma}\rightarrow 0$ and hence
(\ref{eq_limi}) is established.
Inequality (\ref{eq_less}) now implies that
$\|\chi_{E^c\cap \kappa_n} h_0\|_{\Gamma}\rightarrow 0$.
Choosing $n_0$ such that $\|\chi_{E^c\cap \kappa_n} h_0\|_{\Gamma} < \frac{\epsilon}{2}$ for $n\geq n_0$,
we see that
\begin{eqnarray*}
\|\chi_{\kappa_n}h\|_{\Gamma} & \leq & \|\chi_{\kappa_n}(h - h_0)\|_{\Gamma} + \|\chi_{\kappa_n}h_0\|_{\Gamma} \\
& = & \|\chi_{\kappa_n}(h - h_0)\|_{\Gamma} + \|\chi_{E^c\cap\kappa_n}h_0\|_{\Gamma} < \epsilon
\end{eqnarray*}
whenever $n\geq n_0$, which establishes (\ref{eq_con0}).

Now suppose that $\kappa$ is an operator synthetic set and let $h\in \Psi_{\min}(\kappa)$.
We have that
$$h = \chi_{\kappa_n} h + \chi_{\sigma_n\cap \kappa_n^c} h + \chi_{(\kappa_n\cup\sigma_n)^c} h.$$
Since $h$ vanishes on $\kappa$ and $\sigma_n\subseteq \kappa$, we
have that $\chi_{\sigma_n\cap \kappa_n^c} h = 0$. On the other hand,
(\ref{eq_con0}) implies that $\|\chi_{\kappa_n}
h\|_{\Gamma}\rightarrow_{n\rightarrow\infty} 0$. It follows that
$\|h - \chi_{(\kappa_n\cup\sigma_n)^c}
h\|_{\Gamma}\rightarrow_{n\rightarrow\infty} 0$. However,
$\chi_{(\kappa_n\cup\sigma_n)^c}$ is a Schur multiplier vanishing on
the $\omega$-open neighbourhood $\kappa_n\cup \sigma_n$ of $\kappa$.
It follows that $\kappa$ is strong operator Ditkin.
\end{proof}

Since nest algebra bimodules are $\frak{I}$-decomposable, Theorem \ref{th_uadm} yields the following immediate
corollary.

\begin{corollary}\label{c_nestsod}
Every nest set is strong operator Ditkin.
\end{corollary}

Theorem \ref{th_uadm} also implies the following fact obtained in \cite{lt}.

\begin{corollary}\label{c_nestsod}
Every ternary set is strong operator Ditkin.
\end{corollary}

\section{The structure of approximately $\frak{I}$-injective masa-bimodules}\label{s_iam}

In this section, we develop some further operator synthetic
properties of approximately $\frak{I}$-injective masa-bimodules. We
do not know whether the supports of such masa-bimodules are operator
synthetic. However, we show in Theorem \ref{th_decai} below that
$\frak{M}_{\max}(\kappa)$ and $\frak{M}_{\min}(\kappa)$ contain the
same compact operators. Our first aim is to establish a structure
result for approximately $\frak{I}$-injective masa-bimodules
(Theorem \ref{th_decai}). We recall that $\cl K = \cl K(H_1,H_2)$ is
the set of compact operators and $\cl C_2 = \cl C_2(H_1,H_2)$ is the
Hilbert-Schmidt operator ideal; we denote the Hilbert-Schmidt norm
by $\|\cdot\|_2$.

\begin{lemma}\label{l_chsc}
Let $(\cl E_n)_{n\in \bb{N}}$ be a uniformly bounded sequence of Schur idempotents
such that $\ran\cl E_{n+1}\subseteq \ran \cl E_n$, $n\in \bb{N}$.

(i) If $K\in \cl C_2$ the sequence $(\cl E_n(K))_n$ converges in the Hilbert-Schmidt norm.

(ii) If $K\in \cl K$ the sequence $(\cl E_n(K))_n$ converges in the operator norm.
\end{lemma}
\begin{proof}
(i)
Suppose that $\kappa_n\subseteq X\times Y$ is an $\omega$-closed set with $\cl E_n = S_{\chi_{\kappa_n}}$.
It is easy to see that
if $T_{\nph}$ is a Hilbert-Schmidt operator with integral kernel $\nph \in L^2(X\times Y, m\times n)$ then
$\cl E_n(T_{\nph}) = T_{\chi_{\kappa_n}\nph}$. It follows that the sequence $(\cl E_n|_{\cl C_2})_{n\in \bb{N}}$
is a decreasing sequence of orthogonal projections on the Hilbert space $\cl C_2$.
It follows that the sequence $(\cl E_n(T_\nph ))_{n\in \bb{N}}$ converges in the
Hilbert-Schmidt norm for all $\nph \in L^2(X\times Y).$

(ii) Let $K\in \cl K$ and $\epsilon >0.$ There exists $L\in \cl C_2$ such that $\|K-L\|<\frac{\epsilon}{3C}$.
By (i), the sequence $(\cl E_n(L))_n$ converges in $\|\cdot \|_2$ norm, so there exists $n_0$ such that
$$\|\cl E_n(L)-\cl E_m(L)\|_2 < \epsilon/3, \ \ \ n,m\geq n_0.$$
We have
\begin{align*} &\|\cl E_n(K)-\cl E_m(K)\|\\ \leq & \|\cl E_n(K)-\cl E_n(L)\|+ \|\cl E_n(L)-\cl E_m(L)\|+
\|\cl E_m(L)-\cl E_m(K)\| \\ \leq & C\|K-L\|+\|\cl E_n(L)-\cl E_m(L)\|_2+C\|L-K\|< \epsilon,
\end{align*}
for all $n,m\geq n_0$, and the sequence $(\cl E_n(K))_{n\in \bb{N}}$
converges in the operator norm.
\end{proof}

In some of the results that follow, we will use the notion of a
pseudo-integral operator introduced by W. B. Arveson in \cite{a}.
Let $\bb{A}(X,Y) = \bb{A}(X,Y,$ $m,n)$ be the space of Borel
measures $\mu $ on $Y\times X$ of finite total variation for which
there exists a constant $c > 0$ such that
$$|\mu |_X \leq cm \ \mbox{ and } \ |\mu |_Y \leq cn,$$
where $|\mu|$ is the variation of $\mu$ and, for a measure $\nu$ on
$Y\times X$, we denote by $\nu_X$ (resp. $\nu_Y$) the marginal
measure on $X$ (resp. $Y$) given by $\nu_X(\alpha) = \nu(Y\times
\alpha)$ (resp. $\nu_Y(\beta) = \nu(\beta\times X)$). We denote the
smallest constant $c > 0$ with these properties by $\|\mu\|$. To
every $\mu \in \bb{A}(X,Y)$, there corresponds an operator $T_\mu $
satisfying
$$\sca{T_\mu f,g} = \int_\cl {X\times Y}f(x)\overline{g(y)}d\mu(y,x), \ \ \ \ f\in H_1, g\in H_2.$$
The operator $T_{\mu}$ is called the \emph{pseudo-integral operator}
associated with the measure $\mu$. Moreover \cite{tt}, if
$\kappa\subseteq X\times Y$ is an $\omega$-closed set and
$\hat{\kappa} = \{(y,x)\in Y\times X : (x,y)\in \kappa\}$, then
\begin{equation}\label{eq_minp}
\frak{M}_{\min}(\kappa) = \overline{\{T_{\mu} : \mu \mbox{ is supported on } \hat{\kappa}\}}^{w^*}.
\end{equation}

\begin{theorem}\label{th_decai}
Let $\cl M$ be an approximately $\frak{I}$-injective masa-bimodule, and
$\kappa\subseteq X\times Y$ be the $\omega$-closed set with $\cl M = \frak{M}_{\max}(\kappa)$.
There exist an $\frak{I}$-injective masa-bimodule $\cl M_{\rm inj}$ and an
approximately $\frak{I}$-injective masa-bimodule $\cl M_{\rm pai}$ such that
\begin{itemize}
\item[(a)] we have a direct sum decomposition $\frak{M}_{\max}(\kappa) = \cl M_{\rm inj} + \cl M_{\rm pai}$,
\item[(b)] $\cl M_{\rm inj} = \overline{\cl M\cap \cl K}^{w^*}$, and
\item[(c)] $\cl M_{\rm pai}$ contains no non-zero compact operators.
\end{itemize}
Moreover, $\cl M_{\rm inj}$ is the maximal $\frak{I}$-injective masa-bimodule contained in $\cl M$ and
$$\frak{M}_{\min}(\kappa)\cap \cl K = \frak{M}_{\max}(\kappa )\cap \cl K = \overline{
\frak{M}_{\max}(\kappa )\cap \cl C_2} ^{\|\cdot\|}.$$
\end{theorem}
\begin{proof}
By Lemma \ref{l_chsc} (ii), the limit $\Phi(K) =
\lim_{n\rightarrow\infty} \cl E_n(K)$ exists for every $K\in \cl K$.
The mapping $\Phi : \cl K\rightarrow\cl K$ is clearly linear and,
since the mappings $\cl E_n$ are uniformly bounded, it is bounded.
It is also a masa-bimodule idempotent since the $\cl E_n$'s are
such. Passing to the second dual, we obtain an extension (denoted in
the same way) $\Phi : \cl B(H_1,H_2)\rightarrow\cl B(H_1,H_2)$ which
is a Schur idempotent. By paragraph 1.6.1 of \cite{blm}, there
exists a bounded (not necessarily weak* continuous) idempotent $\cl
E$ such that $\ran \cl E = \cl M$. Set $\cl M_{\rm inj} = \ran\Phi$
and $\cl M_{\rm pai} = \ran \Phi^{\perp}\cl E$. Since $\cl E =
\Phi\cl E + \Phi^{\perp}\cl E = \Phi + \Phi^{\perp}\cl E$, we have a
direct sum decomposition $\cl M = \cl M_{\rm inj} + \cl M_{\rm
pai}$. If $K \in \cl M\cap \cl K$ then $\cl E_n(K) = K$ for all $K$
and hence $K = \Phi(K)\in \cl M_{\rm inj}$. Conversely, if $T\in \cl
M_{\rm inj}$, let $T = $w$^*$-$\lim K_i$, for some net $(K_i)$ of
compact operators. But then $T = $w$^*$-$\lim \Phi(K_i) \in \cl
M\cap \cl K$, and (b) follows. Finally, if $K\in \cl M_{\rm pai}\cap
\cl K$ then $K = \Phi^{\perp}\cl E(K) = \cl E\Phi^{\perp}(K) = 0$,
and (c) follows.

Suppose that $\cl M_0$ is an $\frak{I}$-injective masa-bimodule
contained in $\cl M$. By (b) and Proposition \ref{l_iref} (i),
$$\cl M_0\subseteq \overline{\cl M_0\cap \cl K}^{w^*}\subseteq \overline{\cl M\cap \cl K}^{w^*} = \cl M_{\rm inj},$$
hence $\cl M_{\rm inj}$ is a maximal $\frak{I}$-injective masa-bimodule contained in $\cl M$.

The inclusion $\frak{M}_{\min}(\kappa)\cap \cl K= \frak{M}_{\max}(\kappa )\cap \cl K$ is trivial,
while the inclusion
$\frak{M}_{\max}(\kappa )\cap \cl K = \overline{\frak{M}_{\max}(\kappa )\cap \cl C_2} ^{\|\cdot\|}$
follows from Proposition \ref{l_iref} (i).
Suppose that $K\in \frak{M}_{\max}(\kappa )\cap \cl C_2$; then
$K$ is a pseudo-integral operator \cite{a} and hence, by (\ref{eq_minp}), belongs to
$\frak{M}_{\min}(\kappa )\cap \cl K$.
Since the latter space is norm closed, the equalities follow.
\end{proof}

\noindent {\bf Remarks (i) } The subscript of $\cl M_{\rm pai}$
stands for \lq\lq purely approximately $\frak{I}$-injective''. Note that $\cl
M_{\rm pai}$ does not contain a non-zero $\frak{I}$-injective masa-bimodule.
Examples of such masa-bimodules include the ternary masa-bimodules containing no
non-zero rank one operators, in particular continuous masas.

\smallskip

\noindent {\bf (ii) } It follows from Theorem \ref{th_decai} that
an approximately $\frak{I}$-injective masa-bimodule $\cl M$ is injective if and only if
$\overline{\cl M\cap \cl K}^{w^*} = \cl M$. Indeed, one direction follows from Proposition \ref{l_iref} (i); to see the
other, assume that $\overline{\cl M\cap \cl K}^{w^*} = \cl M$ and let $\cl K\ni K_n\rightarrow_{w^*} T\in \cl M_{\rm pai}$.
Letting $\Phi$ be the Schur idempotent with range $\cl M_{\rm inj}$, we have $\Phi^{\perp}(K_n)\rightarrow T$.
However, $\Phi^{\perp}(K_n)\in \cl K\cap \cl M_{\rm pai} = \{0\}$, and hence $T = 0$.

\smallskip

\noindent {\bf (iii) } We note that $\frak{I}$-decomposable
masa-bimodules do not in general contain a maximal $\frak{I}$-injective
masa-bimodule.
For an example, let $\cl A$ be the Volterra nest algebra
acting on $L^2(0,1)$. Then $\cl A = \frak{M}_{\max}(\Delta)$, where $\Delta = \{(x,y)\in [0,1]\times [0,1] : x \leq y\}$.
Let $\Delta_0 = \{(x,y)\in [0,1]\times [0,1] : x < y\}$; then $\Delta_0$ is an $\omega$-open set
which contains every measurable rectangle $\alpha\times\beta$ marginally contained in $\Delta$. Suppose that
$\cl M\subseteq \cl A$ is an $\frak{I}$-injective masa-bimodule, say $\cl M = \frak{M}_{\max}(\kappa)$,
for some $\omega$-open and $\omega$-closed set $\kappa\subseteq [0,1]\times [0,1]$.
It follows that, up to marginal equivalence, $\kappa\subseteq \Delta_0$.
Moreover, $\kappa$ is not marginally equivalent to $\Delta_0$ since $\Delta_0$ is not $\omega$-closed.
But then $\kappa^c\cap \Delta_0$ is a non-marginally null $\omega$-open set and therefore
contains a measurable rectangle $\alpha\times\beta$. It is easy to see that
$\cl M + \frak{M}_{\max}(\alpha\times\beta)$ is an $\frak{I}$-injective masa-bimodule; clearly, it
contains properly $\cl M$. We thus showed that $\cl A$ does not contain a maximal $\frak{I}$-injective
masa-bimodule.

\medskip

\noindent {\bf Example } We present an example of an approximately
$\frak{I}$-injective, but not $\frak{I}$-injective, masa bimodule
$\cl U$, for which the uniform bound for the norms of any sequence
of Schur idempotents with decreasing ranges whose intersection is
$\cl U$ can not be chosen to be smaller than $\frac{2}{\sqrt{3}}$.
Let $\cl M$ and $\cl N$ be weak* closed ternary masa-bimodules, and
let $(\Phi_n)_{n\in \bb{N}}$ (resp. $(\Psi_n)_{n\in \bb{N}}$) be a
sequence of Schur idempotents of norm one with decreasing ranges
such that $\cap \ran\Phi_n = \cl M$ (resp. $\cap \ran\Psi_n = \cl
N$). Let $\Theta_n = \Phi_n + \Psi_n - \Phi_n\Psi_n$, $n\in \bb{N}$.
Then $\Theta_n$ is a Schur idempotent with $\|\Theta_n\|\leq 2$,
$n\in \bb{N}$. We claim that $\cap \ran\Theta_n =  \cl M + \cl N$.
Indeed, write $\cl M_n = \ran\Phi_n$, $\cl N_n = \ran\Psi_n$ and
suppose that $T\in \cap_{n\in \bb{N}} \ran\Theta_n$. For each $n\in
\bb{N}$, write $T = X_n + Y_n$ with $X_n \in \cl M_n$ and $Y_n\in
\cl N_n$. Choose a subsequence $(n_k)_{k\in \bb{N}}$ such that
$\Phi_{n_k}(T) = X_{n_k} + \Phi_{n_k}(Y_{n_k})$ converge weak* along
$k\in \bb{N}$ to an operator $X$. Clearly, $X\in \cl M$. Since $\cl
N$ is invariant under Schur maps, $T - (X_{n_k} +
\Phi_{n_k}(Y_{n_k})) = Y_{n_k} - \Phi_{n_k}(Y_{n_k})$ converge
weak*, along $k\in \bb{N}$, to an operator $Y$ in $\cl N$. Thus, $T
= X + Y\in  \cl M + \cl N$. We showed that $\cap_{n\in \bb{N}}
\ran\Theta_n\subseteq \cl M + \cl N$; the converse inclusion is
trivial.

Now suppose additionally that $\cl M$ is $\frak{I}$-injective while $\cl N$ is not,
$\cl M\cap \cl N = \{0\}$, and $\cl M + \cl N$ is not a TRO (for example, let $H = L^2(0,1)$,
$P$ be the projection onto $L^2(0,\frac{1}{2})$, $\cl M = \cl B(P(H)^\bot,P(H)))$ and
$\cl N$ be the multiplication masa of $L^{\infty}(0,1)$).
From the first paragraph, $\cl M + \cl N$ is an approximately $\frak{I}$-injective masa-bimodule.
Since $\cl N$ is not $\frak{I}$-injective, it does not contain non-zero compact operators and we see that
$\cl M$ is the $\frak{I}$-injective part of $\cl M + \cl N$, while $\cl N$ is its purely approximately $\frak{I}$-injective part.
It follows that $\cl M + \cl N$ is not $\frak{I}$-injective. We claim that for every sequence
$(\cl E_n)_{n\in \bb{N}}$ of Schur idempotents with $\cap_{n\in \bb{N}} \ran\cl E_n = \cl M + \cl N$,
we have that $\|\cl E_n\| \geq \frac{2}{\sqrt{3}}$ eventually. Indeed, if not then, by \cite{kp},
$\ran\cl E_n$ would be a TRO for infinitely many $n$, and hence $\cl M + \cl N$ would be a TRO, contradicting
our assumption.

\smallskip

Since approximately $\frak{I}$-injective masa-bimodules are
$\frak{I}$-decomposable, Theorem \ref{th_uadm} implies that the
support of an approximately $\frak{I}$-injective masa-bimodule is
operator synthetic if and only if it is strong operator Ditkin. In
Theorem \ref{th_syn}, we give a more precise statement for this
special case. We need a couple of preliminary statements.

\begin{lemma}\label{l_schurps}
Let $w : X\times Y\rightarrow \bb{C}$ be a bounded measurable function and $\mu \in \bb{A}(X,Y)$.
Then the measure $w\mu$ given by $w\mu(E) = \int_E w(x,y)d\mu(y,x)$ belongs to $\bb{A}(X,Y)$.
\end{lemma}
\begin{proof}
For a measurable subset $E\subseteq Y\times X$ we have
(the supremum being taken over all partitions $E = \cup_{i=1}^k E_i$)
\begin{eqnarray*}
|w\mu|(E) & = & \sup \sum_{i=1}^k |w\mu(E_i)| = \sum_{i=1}^k \left|\int_{E_i} w(x,y)d\mu(y,x)\right|\\
& \leq & \|w\|_{\infty} \sum_{i=1}^k |\mu|(E_i) = \|w\|_{\infty}
|\mu|(E).
\end{eqnarray*}
Thus, $|w\mu|\leq \|w\|_{\infty} |\mu|$ and hence $|w\mu|_X\leq \|w\|_{\infty} |\mu|_X$ and
$|w\mu|_Y\leq \|w\|_{\infty} |\mu|_Y$. Since $\mu\in \bb{A}(X,Y)$, we have that $w\mu\in \bb{A}(X,Y)$.
\end{proof}

\begin{proposition}\label{p_cuts}
Let $\sigma\subseteq X\times Y$ be an operator synthetic set and $\kappa\subseteq X\times Y$
be an $\omega$-closed set such that $\chi_{\kappa}$ is a Schur multiplier. Then $\sigma\cap \kappa$ is
operator synthetic.
\end{proposition}
\begin{proof}
Since $\chi_{\kappa}$ is a Schur multiplier, $\kappa$ is
$\omega$-open \cite{kp}; assume, without loss of generality, that
$\kappa = \cup_{i=1}^{\infty} \alpha_i\times \beta_i$, where
$\alpha_i\subseteq X$ and $\beta_i\subseteq Y$ are measurable. As in
the proof of Proposition \ref{l_iref} (i), let $(X_N)_{N\in \bb{N}}$
(resp. $(Y_N)_{N\in \bb{N}}$) be an increasing sequence of
measurable subsets of $X$ (resp. $Y$) such that $m(X\setminus X_N) <
1/N$, $n(Y\setminus Y_N) < 1/N$, and $\kappa_N \stackrel{def}{=}
\kappa\cap (X_N\times Y_N)$ is contained in the union of finitely
many of the sets $\alpha_i\times \beta_i$. Write $\kappa_N =
\cup_{j=1}^{k_N} \gamma_j^N \times \delta_j^N$, as a disjoint union,
where each $\gamma_j^N\times\beta_j^N$ is contained in some
$\alpha_i\times \beta_i$. Let $h = f\otimes g \in \Gamma(X,Y)$ be an
elementary tensor, where $f\in L^{\infty}(X,m)$ and $g\in
L^{\infty}(Y,n)$. Then $\|h - (\chi_{X_N}f) \otimes
(\chi_{Y_N}g)\|_{\Gamma}\rightarrow_{N\rightarrow \infty} 0$.

Let $\mu\in \bb{A}(X,Y)$. Then $\chi_{\kappa_N}(x,y)\rightarrow \chi_{\kappa}(x,y)$ for
all $(x,y) \in (\cup_{N=1}^{\infty} X_N)\times (\cup_{N=1}^{\infty} Y_N)$. Since
$((\cup_{N=1}^{\infty} X_N)\times (\cup_{N=1}^{\infty} Y_N))^c$ is marginally null, we have that it is
$\mu$-null, and hence $\chi_{\kappa_N}(x,y)\rightarrow \chi_{\kappa}(x,y)$ for $\mu$-almost
all $(x,y) \in X\times Y$. Using Lebesgue's Dominated Convergence Theorem and Lemma \ref{l_schurps}, we have that
\begin{eqnarray*}
\langle S_{\chi_{\kappa}}(T_{\mu}),h\rangle
& = & \lim_{N\rightarrow \infty} \langle S_{\chi_{\kappa}}(T_{\mu}),(\chi_{X_N}f) \otimes (\chi_{Y_N}g)\rangle\\
& = & \langle T_{\mu},\chi_{\kappa}(\chi_{X_N}f) \otimes (\chi_{Y_N}g)\rangle
= \langle T_{\mu},\chi_{\kappa_N}(f \otimes g)\rangle\\
& = & \sum_{j=1}^{k_N} \langle T_{\mu}, (\chi_{\gamma_j^N}f) \otimes (\chi_{\delta_j^N}g)\rangle\\
& = & \sum_{j=1}^{k_N} \int_{\delta_j^N\times\gamma_j^N} g(y)f(x) d\mu(y,x)\\
& = & \int_{Y\times X} \chi_{\kappa_N}(x,y)g(y)f(x) d\mu(y,x)\\
%\end{eqnarray*}
%\begin{eqnarray*}
& \rightarrow &  \int_{Y\times X} \chi_{\kappa}(x,y)g(y)f(x) d\mu(y,x)\\
& = & \langle T_{\chi_{\kappa}\mu},h\rangle.
\end{eqnarray*}
It follows that $S_{\chi_{\kappa}}(T_{\mu}) = T_{\chi_{\kappa}\mu}$.

Now let $T\in \frak{M}_{\max}(\sigma\cap\kappa)$. Since $\kappa$ is
operator synthetic, (\ref{eq_minp}) implies that $T =$ w$^*$-$\lim
T_{\mu_{\alpha}}$, where $\mu_{\alpha}\in \bb{A}(X\times Y)$ are
measures supported on $\hat{\sigma}$. Since the measure
$\chi_{\kappa}\mu_{\alpha}$ is supported on
$\hat{\sigma}\cap\hat{\kappa}$, we have that
$T_{\chi_{\kappa}\mu_{\alpha}} =
S_{\chi_{\kappa}}(T_{\mu_{\alpha}})$ is a pseudo-integral operator
in $\frak{M}_{\min}(\sigma\cap\kappa)$, Moreover, $T =
S_{\chi_{\kappa}}(T) = $ w$^*$-$\lim
S_{\chi_{\kappa}}(T_{\mu_{\alpha}})$, and hence $T\in
\frak{M}_{\min}(\sigma\cap\kappa)$. Thus,
$\frak{M}_{\max}(\sigma\cap\kappa) =
\frak{M}_{\min}(\sigma\cap\kappa)$.
\end{proof}

\begin{theorem}\label{th_syn}
Let $\cl M$ be an approximately $\frak{I}$-injective masa-bimodule,
$\cl M = \cl M_{\rm inj} + \cl M_{\rm pai}$ be the decomposition from Theorem \ref{th_decai},
$\kappa$ be the support of $\cl M$ and $\kappa_{\rm pai}$ be the support of $\cl M_{\rm pai}$.
The following are equivalent:

(i) \ \ $\kappa$  is operator synthetic;

(ii) \ $\kappa$ is strong operator Ditkin;

(iii) $\kappa_{\rm pai}$ is operator synthetic;

(iv) \ $\kappa_{\rm pai}$ is strong operator Ditkin.
\end{theorem}
\begin{proof}
(i)$\Leftrightarrow$(ii) is a direct consequence of Theorem \ref{th_uadm} and Remark (i)
before Theorem \ref{th_adr}.

The equivalence (iii)$\Leftrightarrow$(iv) follows from the equivalence (i)$\Leftrightarrow$(ii)
upon noticing that $\kappa_{\rm pai}$ is the support of the approximately $\frak{J}$-injective
masa-bimodule $\Phi^{\perp}(\frak{M}_{\max}(\kappa))$, where $\Phi$ is the Schur idempotent with range
$\cl M_{\rm inj}$.

(i)$\Rightarrow$(iii)
Let $\kappa_{\rm inj}$ be the support of $\cl M_{\rm inj}$. By Proposition \ref{p_cuts},
$\kappa_{\rm pai} = \kappa \cap \kappa_{\rm inj}^c$ is operator synthetic.

(iii)$\Rightarrow$(i)
Suppose that $T\in \cl M$. Then $T = T_1 + T_2$, where $T_1\in \cl M_{\rm inj} = \frak{M}_{\min}(\kappa_{\rm inj})$
and $T_2\in \cl M_{\rm pai}$. Since $\kappa_{\rm pai}$ is synthetic,
$T_2\in \frak{M}_{\min}(\kappa_{\rm pai})$, and hence $T_1 + T_2\in \frak{M}_{\min}(\kappa)$.
\end{proof}

We finish this section with another structure result.

\begin{proposition}\label{genarator}
Let $\cl M$ be an approximately $\frak{J}$-injective masa-bimodule
and $(\cl E_n)_{n\in \bb{N}}$ be an associated sequence of Schur
idempotents. Let $\kappa$ be the support of $\cl M$ and $\kappa_n$
be the support of $\ran\cl E_n$, $n\in \bb{N}$. There exists
pseudo-integral operator $T\in \cl B(H_1,H_2)$ such that the limit
$S \stackrel{def}{=}\|\cdot\|$-$\lim_{n\rightarrow\infty}\cl E_n(T)$
exists and
$$\cl B(H_1,H_2)= \overline{[\cl D_2T\cl D_1]}^{w^*},
\frak{M}_{\max}(\kappa_n ) = \overline{[\cl D_2\cl E_n(T)\cl
D_1]}^{w^*}, \frak{M}_{\min}(\kappa) = \overline{[\cl D_2 S \cl
D_1]}^{w^*}.$$
\end{proposition}
\begin{proof}
After a suitable unitary equivalence, we may assume that $m(X) =
n(Y) = 1$. Write $\Ball(\cl B(H_1,H_2))$ for the unit ball of $\cl
B(H_1,H_2)$. Suppose that
$$\Ball(\cl B(H_1,H_2) ) = \overline{\{T_{\phi _k}: k\in \mathbb{N} \}}^{w^*} $$ where $T_{\phi _k}$
is the Hilbert Schmidt operator with integral kernel $\phi _k\in
L^2(X\times Y)$ and
$$\frak{M}_{\min}(\kappa) = \overline{\{T_{\mu _k}: k\in \mathbb{N} \}}^{w^*} $$ where
the measures $\mu_k\in \bb{A}(X,Y)$ are supported on $\hat{\kappa}$.
We define
$$\phi =\sum_{n=1}^\infty \frac{1}{2^n}\frac{|\phi _n|}{\|\phi _n\|_2} \ \mbox{ and } \
\mu =\sum_{n=1}^\infty \frac{1}{2^n}\frac{|\mu _n|}{\|\mu  _n\|}.$$
It is readily checked that $\mu$ is a positive measure in
$\bb{A}(X,Y)$ with $\mu(Y\times X) \leq 1$ and $\|\mu\|\leq 1$. By
Lemma \ref{l_chsc} (ii), there exists an operator $X\in \cl K$ such
that $\|\cdot\|\mbox{-}\lim_{n \rightarrow\infty}\cl E_n(T_{\phi}) =
X$. Thus,
$$\|\cdot\|\mbox{-}\lim_{n\rightarrow\infty}\cl E_n(T_\mu +T_\phi ) = T_\mu + X = S.$$
We write $T = T_\mu + T_\phi$. Let $\alpha\subseteq X$,
$\beta\subseteq Y$ be measurable subsets such that
$P(\beta)TP(\alpha) = 0$. It follows that if $\xi \in H_1, \eta \in
H_2$ are non-negative functions then
$$\int_{\alpha\times \beta}\phi (x,y) \xi (x)\eta (y) dm\times n + \int_{\alpha\times \beta}\xi (x)\eta (y) d\mu =0.$$

Since $\phi \geq 0$ and $\mu \geq 0$ we have that
$$\int_{\alpha\times \beta}\phi_k (x,y) \xi (x)\omega (y) dm\times n= 0, \ \ \ \forall \;\;k\in \bb{N}.$$
We conclude that $\sca{P(\beta)T_{\phi _k}P(\alpha)\xi, \eta } =0$
and hence $\sca{P(\beta)A P(\alpha)\xi, \eta } =0$, for all $A\in
\cl B(H_1,H_2)$. Since $\xi $ and $\eta$ where arbitrary
non-negative functions, we have that $P(\beta)AP(\alpha)=0$ for all
$A\in \cl B(H_1,H_2).$ We thus proved that the reflexive hull of the
space $[\cl D_2T\cl D_1]$ is $\cl B(H_1,H_2).$ Since $\cl
B(H_1,H_2)$ is synthetic, it follows that $\cl B(H_1,H_2) =
\overline{[\cl D_2T\cl D_1]}^{w^*}$. Similarly, if $\xi, \eta $ are
arbitrary non-negative functions such that
$\sca{P(\beta)SP(\alpha)\xi, \eta }=0$ then
$$\sca{P(\beta)T_\mu P(\alpha)\xi, \eta }+ \sca{P(\beta)XP(\alpha)\xi, \eta  }=0.$$
Since $\sca{X\xi', \eta'}\geq 0$ for all non-negative functions
$\xi^\prime, \eta^\prime $, we have that
$$\sca{P(\beta)T_\mu P(\alpha)\xi, \eta }=0,$$ and so $$\sca{P(\beta)A P(\alpha)\xi, \eta
}=0, \;\;\; \mbox{ for all } A\in \frak{M}_{\min}(\kappa).$$ It
follows that $\mathrm{Ref}(\cl D_2S\cl D_1)=\frak{M}_{\max}(\kappa)$
and since $S$ is a pseudo-integral operator, we conclude that
$\frak{M}_{\min}(\kappa) = \overline{[\cl D_2 S \cl D_1]}^{w^*}$.
\end{proof}

\section{Converse results}\label{s_convth}

Let $\cl M$ be an approximately $\frak{I}$-injective $\cl D_2,\cl
D_1$-bimodule, $\cl U$ be a weak* closed $\cl D_2,\cl D_1$-bimodule
and $\cl W = \cl U+\cl M$. By Corollary \ref{c_corm}, $\cl W$ is a
weak* closed masa-bimodule. Moreover, by Corollary \ref{c_corm} and
Proposition \ref{p_cws}, $\cl W$ is reflexive (resp. synthetic) if
$\cl U$ is reflexive (resp. synthetic). In the following, we
consider the converse question: when does the reflexivity (resp.
synthesis) of $\cl W$ imply the reflexivity (resp. synthesis) of
$\cl U?$ This is not true in general. For an example, take $\cl M =
\cl B(H_1,H_2)$ and a non-reflexive (resp. non-synthetic) $\cl
D_2,\cl D_1$-bimodule $\cl U$. We show that in certain cases, when
the masa-bimodules $\cl U$ and $\cl M$ are \lq\lq suitably
positioned'', one can obtain positive results.

If $\cl N$ is a weak* closed masa-bimodule, we will say that a weak* closed masa-bimodule $\cl U$ is
$\cl N$-synthetic if $\frak{M}_{\min}(\kappa)\cap \cl N = \frak{M}_{\max}(\kappa) \cap \cl N$,
where $\kappa$ is the (unique up to marginal equivalence) $\omega$-closed subset of $X\times Y$
with $\frak{M}_{\min}(\kappa)\subseteq \cl U \subseteq \frak{M}_{\max}(\kappa)$.
This notion was introduced and studied in \cite{pp} in the case $X = Y = G$ for some locally compact group $G$.

\begin{theorem}\label{th_rsy}
Let $\cl M$ be an approximately $\frak{I}$-injective masa-bimodule and $\cl U$ be a
weak* closed $\cl M$-synthetic masa--bimodule.
Then

(i) \ $\cl U$ is reflexive if and only if $\cl U + \cl M$ is reflexive;

(ii) $\cl U$ is synthetic if and only if $\cl U + \cl M$ is synthetic.
\end{theorem}
\begin{proof}
(i) By Corollary \ref{c_corm} and Proposition \ref{p_cws},
if $\cl U$ is reflexive (resp. synthetic) then $\cl U + \cl M$ is reflexive (resp. synthetic).
Conversely, assume that $\cl U + \cl M$ is reflexive and let $T\in \Ref\cl U$.
Let $\cl E_n$, $n\in \bb{N}$, be a sequence of Schur idempotents
such that $\cap_{n=1}^{\infty} \ran\cl E_n = \cl M$ and $\|\cl E_n\| \leq C$ for some $C > 0$.
By Corollary \ref{c_corm},
$T\in (\Ref\cl U) + \cl M = \Ref(\cl U + \cl M) = \cl U + \cl M$.
Thus, for each $n\in \bb{N}$, we have that
$T = S_n + M_n$ for some $S_n\in \cl U$ and $M_n\in \cl M_n$.
It follows that $\cl E_n^{\perp}(T) = \cl E_n^{\perp}(S_n) \in \cl U$.
On the other hand, by Proposition \ref{p_intr} (iii),
$\cl E_n(T) = \cl E_n(S_n) + M_n\in \Ref(\cl U) \cap \cl M_n$.
If $S$ is the weak* limit of a subsequence $(\cl E_{n_k}(T))_{k\in \bb{N}}$
then $S\in \Ref(\cl U) \cap \cl M = \cl U \cap \cl M$ since $\cl U$ is $\cl M$-synthetic.
It follows that $$T = \mbox{w*-}\lim(\cl E_{n_k}(T) + \cl E_{n_k}^{\perp}(T)) \in \cl U.$$

(ii) Suppose that $\cl U + \cl M$ is synthetic and let $\cl U_{\min}$ be the minimal weak* closed
masa-bimodule with reflexive cover $\Ref\cl U$. By Corollary \ref{c_corm},
$$\Ref(\cl U_{\min} + \cl M) = (\Ref \cl U_{\min}) + \cl M = \cl U + \cl M.$$
Since $\cl U + \cl M$ is synthetic,
$\cl U_{\min} + \cl M$ (which is weak* closed by Theorem \ref{th_ai} (i)) is reflexive.
By the first paragraph, $\cl U_{\min}$ is reflexive;
thus, $\cl U_{\min} = \Ref\cl U$ and so $\cl U$ is synthetic.
\end{proof}

\begin{theorem}\label{4.1}
Let $\cl M$ be a weak* closed ternary masa-bimodule,
$\cl A_1 = [\cl M^*\cl M]^{-w^*}$ and $\cl A_2 = [\cl M\cl M^*]^{-w^*}$.
Suppose that $\cl U$ is a weak* closed $\cl A_2,\cl A_1$-bimodule. The following hold:

(i) The masa-bimodule $\cl U + \cl M$ is reflexive if and only if $\cl U$ is reflexive.

(ii) The masa-bimodule $\cl U + \cl M$ is synthetic if and only if $\cl U$ is synthetic.
\end{theorem}
\begin{proof}
It follows from \cite[Proposition 2.2]{houston} that (after a unitary equivalence)
$\cl M = \cl M_1\oplus\cl M_2$, where $\cl M_1$ has the form
$\oplus_{k\in \bb{N}} \cl B(H_1^k, H_2^k)$, for some Hilbert spaces $H_1^k$, $H_2^k$,
and $\cl M_2$ is a ternary masa-bimodule which does not contain operators of rank one.
It follows that $\cl A_1$ (resp. $\cl A_2$) contains $\oplus_{k\in \bb{N}}\cl B(H_1^k)$
(resp. $\oplus_{k\in \bb{N}}\cl B(H_2^k)$). Since $\cl U$ is an $\cl A_2,\cl A_1$-bimodule,
we have that
$\cl U = \cl U_1 \oplus \cl U_2$, where
$\cl U_1 = \oplus_{p=1}^{\infty} \cl B(H_1^{k_p}, H_2^{k_p})$
and $\cl U_2$ is a weak* closed $\cl A_2,\cl A_1$-bimodule.

Let $\kappa\subseteq X\times Y$ (resp. $\sigma \subseteq X\times Y$)
be the $\omega$-closed set such that
$\frak{M}_{\min}(\kappa)\subseteq \cl U_2 \subseteq \frak{M}_{\max}(\kappa)$
(resp. $\frak{M}_{\max}(\sigma) = \cl M_2$).
We have that
$$\frak{M}_{\min}(\kappa\cap \sigma)\subseteq \frak{M}_{\min}(\kappa) \cap \cl M_2
\subseteq \frak{M}_{\max}(\kappa) \cap \cl M_2 = \frak{M}_{\max}(\kappa\cap\sigma).$$
By \cite[Theorem 3.6]{houston}, $\kappa\cap \sigma$ is operator synthetic, and hence
$\frak{M}_{\min}(\kappa) \cap \cl M_2 = \frak{M}_{\max}(\kappa) \cap \cl M_2$,
that is, $\cl U_2$ is $\cl M_2$-synthetic.
It follows that $\cl U$ is $\cl M$-synthetic and
the claims follow from Theorem \ref{th_rsy}.
\end{proof}

Theorem \ref{4.1} has the following immediate consequence.

\begin{corollary}\label{4.3} Let $\cl A$ be a CSL algebra, $ \Delta (\cl A) =\cl A\cap \cl A^*$
be the diagonal of $\cl A$, $\cl U$ be a weak* closed subspace
of $\cl A$ which is a $\Delta (\cl A)$-bimodule and such that $\cl A=\cl U+\Delta (\cl A)$.
Then

(i) the space $\cl U$ is reflexive;

(ii) the algebra $\cl A$ is synthetic if and only if the space $\cl U$ is synthetic.
\end{corollary}

\bigskip

\noindent {\bf Acknowledgement} We would like to thank A. Katavolos
for numerous helpful discussions on the topic of this paper.


\begin{thebibliography}{99}


\bibitem{a_ajm}
\textsc{W.B. Arveson}, {\it Analyticity in operator algebras}, {\rm
Amer. J. Math. 89 (1967), 578-642}

\bibitem{a}
\textsc{ W.B. Arveson}, {\it Operator algebras and invariant subspaces},
{\rm Ann. Math. (2) 100 (1974), 433-532}


\bibitem{blm}
{\sc D.P. Blecher and C. Le Merdy}, {\it Operator algebras and
their modules -- an operator space approach}, {\rm Oxford
University Press, 2004}

%\bibitem{drury} {\sc Drury}, {\it On non-triangular sets in tensor algebras},
%{\rm Studia Math.}

\bibitem{eks}  \textsc{J. A. Erdos, , A. Katavolos,  and V. S. Shulman},
\textit{Rank one subspaces of bimodules over maximal abelian
selfadjoint algebras}, \textrm{J. Funct. Anal. 157 (1998) no. 2,
554-587}

\bibitem{ep} {\sc J.A. Erdos and S.C. Power}, {\it Weakly closed ideals of nest algebras},
{\rm J. Operator Theory 7 (1982), no. 2, 219-235}

\bibitem{dav} {\sc K.R. Davidson}, {\it Nest algebras}, {\rm
Longman Scientific and Technical, (1988)}

\bibitem{f}  \textsc{J. Froelich} \textit{Compact Operators, Invariant
Subspaces and Spectral Synthesis}, \textrm{J. Funct. Analysis 81 (1988), 1-37}


\bibitem{haag}
\textsc{U. Haagerup}, \textit{Decomposition of completely bounded
maps on operator algebras}, \textrm{unpublished manuscript}

\bibitem{ht}
\textsc{J.L. Habgood and I.G. Todorov},
\textit{Convergence of bimodules over maximal abelian selfadjoint algebras},
\textrm{J. Operator Theory 63 (2010), no. 3, 301-315}

\bibitem{hayshul}
\textsc{R. G. Haydon and V. S. Shulman}, \textit{On a
measure-theoretical problem of Arveson}, \textrm{Proc. Amer. Math.
Soc. 124 (1996), 497-503}

\bibitem{kp}
\textsc{A. Katavolos and V. Paulsen}, \textit{On the ranges of
bimodule projections}, \textrm{Canad. Math. Bull. 48 (2005) no. 1,
97-111}

\bibitem{kt}  \textsc{A. Katavolos and I.G. Todorov}, \textit{Normalizers of
 operator algebras and reflexivity}, \textrm{Proc. London Math. Soc. (3) 86 (2003), 463-484}


%\bibitem{ks}
%\textsc{E. Kissin and V.S. Shulman}, \textit{Operator multipliers},
%\textrm{Pacific J. Math. 248 (2006), no 2.}

\bibitem{ls}  \textsc{A.I. Loginov and V.S.Shulman}, \textit{Hereditary and
intermediate reflexivity of W*-algebras}, \textrm{Izv. Akad. Nauk SSSR 39
(1975), 1260-1273; Math. USSR-Izv. 9 (1975), 1189-1201}


\bibitem{lt} {\sc J. Ludwig and L. Turowska}, {\it On the connection between sets of
operator synthesis and sets of spectral synthesis for locally compact groups},
{\rm J. Funct. Anal. 233 (2006), 206-227}

\bibitem{pp} {\sc K. Parthasarathy and R. Prakash}, {\it Spectral synthesis and operator synthesis},
{\rm Studia Math. 177 (2006), No. 2, 173-181}

\bibitem{paulsen}
{\sc V. Paulsen}, {\it Completely bounded maps and operator
algebras}, {\rm Cambridge University Press, 2002}

\bibitem{peller} {\sc V.V. Peller,}
{\it Hankel operators in the theory of perturbations of unitary and
selfadjoint operators}, {\rm Functional Anal. Appl. 19 (1985), no.
2, 111-123}

\bibitem{st} {\sc N. Spronk and L. Turowska}, {\it Spectral synthesis and operator synthesis
for compact groups},
{\rm J. London Math. Soc. (2) 66 (2002), 361-376}

\bibitem{vs_ter} {\sc V.S. Shulman}, {\it Multiplication operators and spectral
synthesis}, {\rm Soviet Math. Dokl. 42 (1991), no. 1, 133-137}

\bibitem{st1} {\sc V.S. Shulman and L. Turowska}, {\it Operator
synthesis I. Synthetic sets, bilattices and tensor algebras},
{\rm J. Funct. Anal. 209 (2004), 293-331}

\bibitem{st2} {\sc V.S. Shulman and L. Turowska},
{\it Operator synthesis II: Individual synthesis and linear operator equations} {\rm  J. Reine Angew. Math.  590  (2006), 143--187}

\bibitem{smith}
\textsc{R. R. Smith}, \textit{Completely bounded module maps and
the Haagerup tensor product}, \textrm{J. Funct. Anal. 102 (1991),
156-175}

\bibitem{t_jlms} {\sc I.G. Todorov}, {\it Spectral synthesis and masa-bimodules},
{\rm J. London Math. Soc. (2) 65 (2002), no. 3, 733-744}

\bibitem{houston} {\sc I.G. Todorov}, {\it Synthetic properties of ternary masa-bimodules},
{\rm Houston J. Math. 32 (2006), no. 2, 505-519}

\bibitem{tt}
{\sc I.G. Todorov and L. Turowska}, {\it Closable multipliers on
group algebras}, {\rm preprint}

\end{thebibliography}
\end{document}